%% file: Nonlinear_arXiv_v1.tex
\documentclass[11pt]{article}

\usepackage{cite}
\usepackage{textcomp}

\input{Header}

\usepackage[normalem]{ulem}
\usepackage[toc,page,header]{appendix}
\usepackage{times}
\date{}

\title{Nonlinear Consensus for Distributed Optimization}
\author{Hsu Kao and Vijay Subramanian
\thanks{Hsu Kao is with the Department of Electrical Engineering and Computer Science, University of Michigan, Ann Arbor, MI 48109 USA (e-mail: hsukao@umich.edu).}
\thanks{Vijay Subramanian is with the Department of Electrical Engineering and Computer Science, University of Michigan, Ann Arbor, MI 48109 USA (e-mail: vgsubram@umich.edu).}
}

\begin{document}

\maketitle

\begin{abstract}
Distributed optimization algorithms have been studied extensively in the literature; however, underlying most algorithms is a linear consensus scheme, i.e. averaging variables from neighbors via doubly stochastic matrices. We consider nonlinear consensus schemes with a set of \emph{time-varying} and \emph{agent-dependent} monotonic Lipschitz nonlinear transformations, where we note the operation is a \emph{subprojection} onto the consensus plane, which has been identified via stochastic approximation theory. For the proposed nonlinear consensus schemes, we establish convergence when combined with the NEXT algorithm \cite{Scutari_NEXT_2016}, and analyze the convergence rate when combined with the DGD algorithm \cite{Nedic_DistSubgrad_2009}. In addition, we show that convergence can still be guaranteed even though column stochasticity is relaxed for the gossip of the primal variable, and then couple this with nonlinear transformations to show that as a result one can choose any point within the ``shrunk cube hull" spanned from neighbors' variables during the consensus step. We perform numerical simulations to demonstrate the different consensus schemes proposed in the paper can outperform the traditional linear scheme.
\end{abstract}

\def\keywords{\vspace{.5em}\small
{\textbf{Keywords}:\,\relax%
}}
\def\endkeywords{\par}

\begin{keywords}
Network consensus, distributed optimization, nonlinear transformation, stochastic approximation
\end{keywords}

\section{Introduction}\label{sec:1}
The network consensus problem, which concerns the convergence behavior of a multi-agent network agreeing on the value of a variable, is widely studied \cite{Bauso_Cons_2006,Boyd_Average_2004,Tsitsiklis_Cons_2006,ElChamie_Cons_2015}. In the problem, each node in the underlying directed or undirected network is given an arbitrary initial value, and the nodes exchange their values through the network links to reach a consensus. The literature on the network consensus problem has focussed on the following: choosing the consensus weight matrix to maximize the convergence speed \cite{Boyd_Average_2004,ElChamie_Cons_2015}, analysis of the dependence of the convergence rate on the number of nodes and spectral radius of the weight matrix \cite{Tsitsiklis_Cons_2006}, and the conditions for convergence of different concensus schemes using stability analysis \cite{Bauso_Cons_2006}. Whereas consensus schemes have broad applications, they are critical to \emph{distributed optimization}. Nearly all distributed optimization algorithms are composed of two core steps: the ``descent step" and the ``consensus step." The bulk of the literature on distributed optimization studies a multitude of schemes for the descent process, including primal/dual methods with various acceleration schemes for different types of communication graphs and time delays, and focuses on convergence analysis for different types of objective functions, step schedules, etc \cite{Yang_DistOptSurvey_2019,Zheng_DistOptSurvey_2022}. However, most algorithms use simple linear consensus for the consensus step, so that the role of the consensus scheme in the context of distributed optimization has received less attention. Common criteria for distributed optimization algorithms converging to optimality include diminishing deviation from mean and deviation to optimum, both in decision variable and objective function. The study of a general class of consensus schemes can not only provide speed-up in convergence, but also greater flexibility in selecting time-dependent, agent-dependent, or other criteria focused averaging processes based on designer's preferences.

On the other hand, originating from the seminal work of \cite{Robbins_SA_1951}, the theory of \emph{stochastic approximation} studies the convergence behavior of some prototypical iterative processes that solve certain problems with noisy data samples; common problems that fall into the category include root-finding problems, fixed-point iterations, and optimization problems \cite{Borkar_SABook_2008}. Using the ordinary differential equation (ODE) approach \cite{Meerkov_ODEApproach_1972a,Meerkov_ODEApproach_1972b}, the theory characterizes asymptotic behavior of the \emph{discrete} processes as tracking trajectories of ODEs, which are \emph{continuous}. The concept of \emph{two time-scale} stochastic approximation is proposed in \cite{Borkar_TwoTime_1997}. Here, the fast process and the slow process are coupled together; however, due to different time-scale, the slow process can view the fast process as quasi-static (reaching equilibrium), so that the combined process asymptotically tracks the slow process with iterates confined in the invariant set of the fast process. This concept is extremely useful as it subsumes many important algorithms in the field of optimization, machine learning, and reinforcement learning. In particular, many distributed optimization methods are two time-scale stochastic approximation algorithms \cite{Borkar_SABook_2008}, with the slow process being the descent process, and the fast process being linear subprojections onto the \emph{consensus plane}, i.e. the hyperplane of the ensemble variable where all local variables agree. From this perspective, \cite{Borkar_Gossip_2016} extends the subprojection part in a gossip-based distributed stochastic approximation scheme to nonlinear operations satisfying certain regulatory conditions, and characterizes the limiting ODEs. Another example of two time-scale stochastic approximation algorithms is \cite{Borkar_Project_2018}, where distributed projections are considered in the fast process; this is hard to capture via the nonlinear operations in \cite{Borkar_Gossip_2016}, but is useful when projections onto the intersection of the constraint sets are infeasible. For yet another example, concentration bounds for stochastic approximation algorithms with contractive maps are derived with application to convergence behavior of value functions in reinforcement learning in \cite{Borkar_RLConSA_2021}.

While most distributed optimization algorithms simply alternate between the descent step and the consensus step with each step chosen in an alternating manner, another line of work studies the optimal frequency between them and analyzes the convergence rate specifically for the distributed optimization setting \cite{Wei_MultCons_2018,Wei_MultCons_2021}. In particular, unlike in the literature where a local optimization step (e.g. a local gradient descent step) is performed immediately followed by taking average once intermittently, \cite{Wei_MultCons_2018,Wei_MultCons_2021} suggest multiple applications of the ``descent operator" followed by multiple applications of the ``consensus operator" with certain schedules for convergence speed-up.

In this paper, we aim to answer the following questions: what general class of consensus schemes would be appropriate to be adopted for distributed optimization algorithms? In addition, can we establish convergence to the optimum for these schemes, and contrast their convergence behavior with linear consensus schemes? We start with a class of nonlinear consensus schemes that exploit \emph{nonlinear transformations}, where linear averages are taken after the transformations; we argue why this enriched class could converge faster than linear schemes in pure network consensus problems. Then, enlightened from the nonlinear gossip result of \cite{Borkar_Gossip_2016}, we identify that this class works as a more general form of subprojections onto the consensus plane, and hence can lead to convergent algorithms when placed as the consensus part of distributed optimization algorithms from the perspective of stochastic approximation theory: we establish convergence in the case of the NEXT algorithm \cite{Scutari_NEXT_2016}, and provide convergence rate analysis (in addition to ensuring convergence) in the case of the DGD algorithm \cite{Nedic_DistSubgrad_2009}. We further relax the column stochasticity requirement commonly assumed with the necessary row stochasticity in the literature, which leads to more flexible provably convergent algorithms: we again establish convergence for NEXT, and provide a similar counterpart algorithm for DGD with gradient tracking variables, with its convergence left as future work. Finally, we combine the two ideas to obtain an even more general class of consensus schemes and provide its convergence with NEXT. This more general consensus class along with the family of $p$-means transformations allows us to choose any point in the ``shrunk cube hull" (defined in \pref{sec:5.1}) that ranges from element-wise minimum to element-wise maximum in the consensus step. Our analysis of the convergence of these algorithms along with some numerical examples demonstrates that the convergence rate critically depends on the relation between the consensus scheme and the gradient direction. We then use this as motivation to propose algorithms that \emph{align consensus steps with negative gradient directions}.

Our contribution is threefold. First, we propose two levels of generalizations for consensus schemes in distributed optimization algorithms: one is taking \emph{time-varying} and \emph{agent-dependent} element-wise nonlinear transformations, and the other is relaxing column stochasticity; we combine the two generalizations to conclude that one can choose any point within the shrunk cube hull in the consensus step, which endows great flexibility. Second, we show the convergence of these generalizations can be guaranteed through the example of NEXT, and derive their effect on convergence rate through the example of DGD. Third, we manifest how these generalizations lead to faster convergence through numerical examples, and point out that the better performance comes from the relation between initial variables, objective functions (hence gradient directions), and consensus schemes. Note that although \cite{Borkar_Gossip_2016} provides general results regarding gossip-based distributed stochastic approximation schemes with nonlinear gossip, this work is not subsumed in their framework because of the following. To begin with, we allow \emph{time-varying} and \emph{agent-dependent} subprojections, implying that the nonlinear transformations in our schemes can be varied at every step (using past information). Moreover, with the column stochasticity relaxation, we allow choosing any point in the shrunk convex/cube hull, which cannot be captured by any nonlinear operations in \cite{Borkar_Gossip_2016}, and is not seen in the literature as well. The considered convergence results are also different: \cite{Borkar_Gossip_2016} focuses on almost surely convergence to an internally chain transitive invariant set using the ODE approach, while we study guaranteed convergence to optimum/stationary points.

The paper is organized as follows. In \pref{sec:2} we describe the setting as well as some terminologies for later usage. In \pref{sec:3.1}, we study the convergence behavior of nonlinear consensus schemes that use nonlinear transformations in a pure consensus problem setting (without coupling with optimization); then we show the convergence of NEXT and DGD when these nonlinear transformations are used in the consensus step. We then show the convergence of NEXT without column stochasticity in \pref{sec:4.1}, while the convergence of the DGD algorithm counterpart with gradient tracked by another variable is conjectured in \pref{sec:4.2}. Combining the nonlinear transformation with column stochasticity relaxation, we further enlarge the possible choices for the next iterate in nonlinear consensus schemes from a ``shrunk convex hull" described in \pref{sec:4} to a ``shrunk cube hull" in \pref{sec:5.1}, and propose algorithms that align the consensus step direction to the negative total gradient in \pref{sec:5.2}. Simulations for the numerical examples and conclusion are given in \pref{sec:6} and \pref{sec:7}, respectively.

\section{Models and Assumptions}\label{sec:2}
We describe the standard distributed optimization setup in \pref{sec:2.1}, then present two distributed optimization algorithms -- the most basic DGD algorithm \cite{Nedic_DistSubgrad_2009} and the NEXT algorithm that can handle non-convex objective functions \cite{Scutari_NEXT_2016}, in \pref{sec:2.2} and \pref{sec:2.3}, respectively. Lastly, the consensus plane and the projections onto it are defined in \pref{sec:2.4}, which will be used frequently in this paper.

\subsection{Distributed Optimization Setup}\label{sec:2.1}
In this subsection, we give the system model of distributed optimization and assumptions. Consider a network $\Nc=\{1,\dots,n\}$ that consists of $n$ nodes. We aim to solve an optimization problem of the form
\beq\label{eq:2.1}
\min_{\xB\in\Kc}\quad F(\xB)=\sum_{i=1}^nf_i(\xB),
\eeq
where all $f_i$'s are $C^1$ smooth but can be non-convex. The goal is to let these nodes cooperatively solve the problem in a distributed fashion. Therefore, each $j\in\Nc$ maintains a copy of the entire decision variable $\xB$, referred to as $\xB_j$. Then \eqref{eq:2.1} is equivalent to solving the optimization problem
\beq\label{eq:2.2}
\min_{\xB_j\in\Kc}\quad\sum_{i=1}^nf_i(\xB_j)
\eeq
at each $j\in\Nc$ subject to the constraint that all nodes agree on their optimal choices, i.e., we enforce
\beq\label{eq:2.3}
\xB_1=\xB_2=\dots=\xB_n.
\eeq
In the context of distributed optimization, node $i$ only has the information of $f_i$. It would require communication between the nodes to solve the problem in \eqref{eq:2.2}-\eqref{eq:2.3}.

Below are the standard assumptions on the objective functions and the constraint set.\\
{\bf Assumption A}\\
{\bf (A1)} The set $\Kc\in\R^d$ is closed and convex;\\
{\bf (A2)} $U$ is coercive, that is, $\lim_{\xB\in\Kc,|\xB|\ra\infty}U(\xB)=\infty$; based on this we can effectively assume that $\mathcal{K}$ is compact;\\
{\bf (A3)} $f_i$'s have bounded gradients, i.e. $\exists\es B$ s.t. $\|\gd f_i(\xB)\|\leq B$ for all $\xB\in\Kc$;\\
{\bf (A4)} $f_i$'s have Lipschitz gradients, i.e. $\exists\es L_i$ s.t. $\|\gd f_i(\xB)-\gd f_i(\yB)\|\leq L_i\|\xB-\yB\|$ for all $\xB,\yB\in\Kc$.

The set of nodes $\Nc$ along with a set of undirected edges $\Ec$ form a graph $\Gc=(\Nc,\Ec)$. This graph captures how communications take place -- node $i$ and $j$ can only communicate if $(i,j)\in\Ec$. The communication graph $\Gc$ is assumed to be connected and simple to foster communication between the nodes; otherwise, the problem is generally unsolvable. Moreover, the graph $\Gc$ is associated with a symmetric doubly stochastic matrix $\WB\in\R^{n\times n}$ such that $\|\WB-\frac{1}{n}\1\1^\top\|_2<1$, and that $W_{ij}>0$ only if $(i,j)\in\Ec$ with the exception that $W_{ii}>0$ is allowed even if we assume there is no self-loop.

In the following, we introduce the DGD algorithm and the NEXT algorithm that will later be used to demonstrate the proposed nonlinear consensus algorithms. Note that our algorithms are \emph{meta-schemes} that work with primal distributed optimization algorithms, not just DGD and NEXT. We select DGD for the showcase as it is simple, well-known, and easier to analyze for convergence rate. On the other hand, we select NEXT due to its built-in gradient tracking scheme, so that we can easily relax column stochasticity; this is the focus from \pref{sec:4} onwards.

\subsection{The DGD Algorithm}\label{sec:2.2}
The DGD algorithm first proposed in \cite{Nedic_DistSubgrad_2009} can handle non-smooth objective functions with sub-gradient descent. For simpler convergence analysis, we will consider strongly convex and smooth objective functions similar to \cite{Yuan_DGDConv_2016,Wei_MultCons_2018,Jakovetic_ConvRate_2012}. Specifically, we consider the setup given in \pref{sec:2.1} with a few more constraints: we let $\Kc=\R^d$ to save the hassle of projections; also, we assume $f_i$'s are convex and denote $L\triangleq \Lmax=\max_iL_i$ for simplicity. Just as many primal distributed optimization algorithms, in DGD node $i$ maintains a local copy of the decision variable, and updates the copy by taking the convex combination of the copies in its neighbor set and subtracting its current gradient. The DGD algorithm is presented in \pref{alg:2.3}.

\setcounter{AlgoLine}{0}
\begin{algorithm}
\caption{DGD Algorithm}\label{alg:2.3}
\nl Initialization: $\xB_i[0]\in\R^d$, $t=0$.\\
\nl \While{$\xB[t]$ does not satisfy the termination criterion}{
\nl $t\la t+1$\\
\nl $\xB_i[t+1]=\sum_jW_{ij}\xB_j[t]-\al[t]\gd f_i(\xB_i[t])$\eqnum\label{eq:2.6}
}
\textbf{Output:} $\xB[t]$
\end{algorithm}

Line 4 is performed for all $i\in\Nc$, and $\al[t]$ is the learning rate at time $t$. The initial copy value $\xB_i[0]$ can be chosen arbitrarily from $\R^d$. With the strong convexity assumption, there is a unique optimal solution $\xB^*$, and the convergence rate results depend on whether weighted running average \cite{Jakovetic_ConvRate_2012, Kao_ConvRate_2019} or last iterate convergence \cite{Yuan_DGDConv_2016} is considered.

\subsection{The NEXT Algorithm}\label{sec:2.3}
The NEXT algorithm is a non-convex distributed optimization algorithm that optimizes convex surrogate functions in each iterate for faster convergence, a technique called successive convex approximation \cite{Scutari_NEXT_2016}. In the NEXT algorithm, each node performs a local convex optimization, and then some information will be exchanged in the network. At a high level, just as many primal distributed optimization algorithms, the first step is the ``descent step" and the second is the ``consensus step;" the two steps are iteratively applied to obtain the solution \cite{Scutari_NEXT_2016}. In the first step of time $t$, the node $i$ solves a convex approximation of the whole objective function by convexizing its own objective function $f_i$ \emph{parametrized by} the current iterate $\xB_i[t]$ to be a strongly convex surrogate function $\ft_i(\bullet;\xB_i[t])$, while linearizing the sum of other nodes' objective functions $\sum_{j\neq i}f_j$. The surrogate function of $f_i$ at iterate $\xB$, denoted as $\ft_i(\bullet;\xB)$, should satisfy the following assumptions:
\\
{\bf Assumption F}\\
{\bf (F1)} $\ft_i(\bullet;\xB)$ is convex;\\
{\bf (F2)} $\gd\ft_i(\xB;\xB)=\gd f_i(\xB)$ for all $\xB\in\Kc$;\\
{\bf (F3)} $\ft_i(\bullet;\xB)$'s are coercive for all $\xB\in\Kc$ and $i\in\Nc$.
\\
With the surrogate function, node $i$ solves the following local convex optimization problem at time $t$
\beq\label{eq:2.4}
\xBt_i(\xB_i[t],\piBt_i[t])=\underset{\xB_i\in\Kc}{\arg\min}\ft_i(\xB_i;\xB_i[t])+\piBt_i[t]^\top(\xB_i-\xB_i[t])+G(\xB_i)
\eeq
where $\piBt_i[t]$ is a variable maintained at node $i$ to approximate $\piB_i[t]\triangleq\sum_{j\neq i}\gd f_j(\xB_i[t])$ for the linearization of others' objectives. Since node $i$ does not have this information, it uses another variable $\yB_i[t]$ to track the average of gradients $\frac{1}{n}\sum_{j=1}^n\gd f_j(\xB_i[t])$ with the gossip of similar variables maintained in all the nodes in the network, and approximates $\piB_i[t]$ with $\piBt_i[t]\triangleq n\cdot\yB_i[t]-\gd f_i[t]$ where we denote $\gd f_i[t]\triangleq\gd f_i(\xB_i[t])$. The NEXT algorithm is given in \pref{alg:2.1}. All operations containing index $i$ are performed for all $i\in\Nc$, i.e., Line 1, 5, 6, 7, 9, 10, and 11.

\setcounter{AlgoLine}{0}
\begin{algorithm}
\caption{NEXT Algorithm}\label{alg:2.1}
\nl Initialization: $\xB_i[0]\in\Kc$, $\yB_i[0]=\gd f_i[0]$, $\piBt_i[0]=n\yB_i[0]-\gd f_i[0]$, $t=0$.\\
\nl \While{$\xB[t]$ does not satisfy the termination criterion}{
\nl $t\la t+1$\\
\nl \textit{Local SCA optimization}: for all $i\in\Nc$\\
\nl $\xBt_i[t]=\xBt_i(\xB_i[t],\piBt_i[t])$\\
\nl Find a $\xB_i^{inx}[t]$ s.t. $\|\xBt_i[t]-\xB_i^{inx}[t]\|\leq\ep_i[t]$\\
\nl $\zB_i[t]=\xB_i[t]+\al[t](\xB^{inx}_i[t]-\xB_i[t])$\\
\nl \textit{Consensus update}: for all $i\in\Nc$\\
\nl $\xB_i[t+1]=\sum_{j=1}^nW_{ij}\zB_j[t]$\\
\nl $\yB_i[t+1]=\sum_{j=1}^nW_{ij}\yB_j[t]+(\gd f_i[t+1]-\gd f_i[t])$\\
\nl $\piBt_i[t+1]=n\cdot\yB_i[t+1]-\gd f_i[t+1]$
}
\textbf{Output:} $\xB[t]$
\end{algorithm}

With Assumption A and Assumption F, the result established in \cite{Scutari_NEXT_2016} is convergence to the stationary points (see Theorem 4 of \cite{Scutari_NEXT_2016}), whose definition is given as follows.
\begin{dff}\label{dff:2.1}
A point $\xB^*$ is a stationary point of Problem \eqref{eq:2.1} if $\gd F(\xB^*)^\top(\yB-\xB^*)\geq0$ for all $\yB\in\Kc$.
\end{dff}

\subsection{Definitions Regarding the Ensemble Vector}\label{sec:2.4}
Recall from \pref{sec:2.1} that each node $i\in[n]$ maintains a copy of the decision variable $\xB_i\triangleq[x_{i,1}\es\cdots\es x_{i,d}]^\top\in\R^d$. We denote the ensemble vector consisting of the copies from all the nodes as $\XB\triangleq[\xB_1^\top\es\cdots\es \xB_n^\top]^\top\in\R^{nd}$, and the dimension $l$ part of the ensemble as $\XB_l=[x_{1,l}\es\cdots\es x_{n,l}]^\top\in\R^n$.

\begin{dff}\label{dff:3.1}
The consensus plane $\Cc$ contains all the ensemble vectors such that the variables in all the nodes agree, i.e.
\[
\Cc\triangleq\{\XB=[\xB_1^\top\es\cdots\es\xB_n^\top]^\top\in\R^{dn}:\xB_1=\cdots=\xB_n\}.
\]
\end{dff}

\begin{dff}[from \cite{Borkar_Gossip_2016}]\label{dff:3.2}
A projection $\Pc$ onto a set $\Cc$ is a function such that $\Pc(\Pc(\xB))=\Pc(\xB)$ and $\Pc(\xB)\in\{\xB:\Pc(\xB)=\xB\}$ for all $\xB$. A function $\Fc$ is called a projection subroutine of $\Pc$ if the following holds: (1) $\Fc$ is continuous; (2) $\lim_{n\ra\infty}\Fc^n=\Pc$ uniformly; (3) $\Pc(\Fc(x))=\Fc(\Pc(x))=\Pc(x)$. We will also call this $\Fc$ a subprojection onto $\Cc$.
\end{dff}

\section{Nonlinear Consensus Schemes}\label{sec:3}
From the perspective of stochastic approximation, the linear consensus update is a special type of subprojection onto the consensus plane whereas the ``consensus variable" $\xB_c$, which is any type of ``mean" of $\XB$ (e.g. $\frac{1}{n}\sum_j\xB_j$), descends to the optimum; the reader may refer to \cite{Kao_LocalApprox_2021} (NEXT) and \cite{Borkar_Gossip_2016,Borkar_SABook_2008} (DGD) for analysis from this perspective. Enlightened from the result of \cite{Borkar_Gossip_2016}, which broadens the types of subprojections that can be used, in this section we study nonlinear subprojection methods which transform variables into another domain where linear averages are taken, in particular the family of $(p,w)$-means. We first study the convergence behavior of nonlinear subprojections to the consensus plane in \pref{sec:3.1}, then in \pref{sec:3.2} we describe the combinations of the nonlinear subprojections with NEXT and DGD.

\subsection{Nonlinear Subprojection by Transformation}\label{sec:3.1}
The Line 9 in \pref{alg:2.1} can be seen as a type of subprojection to the consensus plane $\Cc:=\{\XB=[\xB_1^\top\es\cdots\es\xB_n^\top]^\top\in\R^{dn}:\xB_1=\cdots=\xB_n\}$. From the viewpoint of \cite{Borkar_Gossip_2016}, a subprojection $\Fc(\cdot)$ is a function such that under infinitely many applications, the iterate will end up in $\Cc$ for any starting point $\XB[0]$ (see \pref{dff:3.2}). In the distributed optimization setting, in each iteration, a node $i$ takes account of the information from its neighborhood $\{\xB_j:j\in Nb(i)\}$ where $Nb(i):=N(i)\uU\{i\}$ and $N(i):=\{j:(i,j)\in\Ec\}$ and projects that to a ``sub-consensus plane" $\Cc_i:=\{\XB=[\xB_1^\top\es\cdots\es\xB_n^\top]^\top\in\R^{dn}:\xB_j=\xB_k\es\forall\es j,k\in Nb(i)\}$. The connectedness assumption of the communication graph $\Gc$ guarantees that the intersection of these sub-consensus planes $\bdU_i\Cc_i$ is $\Cc$. In \pref{alg:2.1} and \pref{alg:2.3}, for node $i$ this is done by taking a convex combination of all its neighbors' information $\XB_i=[\xB_j]_{j\in Nb(i)}$, which can be thought of as finding the point in $\Cc_i$ that minimizes the weighted sum of squares of distances from the point to each element in $\XB_i$; i.e., the next iterate after the projection subroutine is
\[
\sum_{j\in Nb(i)}W_{ij}\xB_j[t]=\argmin_{\xB\in\R^d}\sum_{j\in Nb(i)}W_{ij}\|\xB-\xB_j[t]\|^2.
\]
We refer to this as the linear consensus scheme.

Now, we consider nonlinear consensus schemes. For simplicity, in this subsection let us consider the one-dimension case ($d=1$) and focus on repeated projection subroutines onto the consensus plane (instead of the distributed optimization context where the descent steps and consensus steps are interlaced together) first. In \cite{Tikhomirov_Mean_1991}, it is proved that under certain ``axioms of mean," the mean of $x_1,\cdots,x_n$ should take the form of $M(x_1,\cdots,x_n)=\vf^{-1}\left(\frac{\vf(x_1)+\cdots+\vf(x_n)}{n}\right)$ where $\vf$ is a continuous increasing function. Since we do not require the symmetry property (i.e., $M(x_1,\cdots,x_n)$ does not have to be a symmetric function), and want to include the weights $W_{ij}$ in the average expression, we consider
\beq\label{eq:3.10}
x_i[t+1]=\vf^{-1}\left(\sum_jW_{ij}\vf(x_j[t])\right),
\eeq
with $\vf$ being a strictly monotonic bijective function on $\R$. Note that node $i$ can only use the information from its neighbors due to the communication constraint imposed by the network, which is automatically applied in \eqref{eq:3.10} via the constraint on $\WB$. If we start with $\xB[0]=[x_1[0]\es\cdots\es x_n[0]]^\top$ and repeatedly apply \eqref{eq:3.10}, then this consensus scheme will converge to
\beq\label{eq:3.11}
x^*=\vf^{-1}\left(\frac{1}{n}\sum_{i=1}^n\vf(x_i[0])\right):=\chi(\xB[0]).
\eeq
We can rewrite each iteration in \eqref{eq:3.10} as $z_i[t+1]\triangleq\vf(x_i[t+1])=\sum_jW_{ij}\vf(x_j[t])$, which simply means taking convex combinations of $z_j\triangleq\vf(x_j)$'s. That is to say, this is a linear consensus scheme in the \textit{transformed domain} of $\vf$. Since the infinite product of $\WB$ converges to $\frac{1}{n}\1_n\1_n^\top$, i.e. $\lim_{t\ra\infty}\WB^t=\frac{1}{n}\1_n\1_n^\top$ \cite{Boyd_Average_2004}, the iterations converge to $\zs=\frac{1}{n}\sum_{i=1}^nz_i[0]$ where $z_i[0]=\vf(x_i[0])$ and from the invertibility of $\vf$ we can find an $\xs$ such that $\zs=\vf(\xs)$.

If we limit the domain of consideration to $\R^+$, a common choice of $\vf$ is $\vf(x)=x^p$:
\beq\label{eq:3.12}
x_i[t+1]=\left(\sum_jW_{ij}x_j^p[t]\right)^{1/p},
\eeq
which is called the $(p,w)$-mean in \cite{Lemmens_NonlinPF_2012}, or mean of order $p$ when $\WB$ is uniform in \cite{Bauso_Cons_2006}. Arithmetic, geometric, and harmonic means fall into this category, corresponding to $p=1$, $p\ra0$, and $p=-1$ when $\WB$ is uniform, respectively. When $p\ra\infty$ and $p\ra-\infty$, \eqref{eq:3.12} also corresponds to the max operation $x_i[t+1]=\max_{j\in Nb(i)}x_j[t]$ and the min operation $x_i[t+1]=\min_{j\in Nb(i)}x_j[t]$ in the support, respectively \cite{Lemmens_NonlinPF_2012}. It is natural to ask whether this consensus scheme leads to faster projection onto $\Cc$, and whether it could be used to increase the convergence speed of distributed optimization algorithms.

We discuss the first question in more depth here, while the second one is studied later in this paper. Note that the ``sample variance" in the transformed domain $\sum_{i=1}^n(z_i[t]-\zs)^2=\sum_{i=1}^n[\vf(x_i[t])-\vf(\zs)]^2$ decreases over time \cite{Nedic_Dist_2009}, and so does $V[t]=\sum_{i=1}^n(x_i[t]-\xs)^2$ due to the monotonicity of $\vf$. Therefore, one reasonable way to evaluate the convergence speed of repeated subprojections to $\Cc$ is by how many iterations it takes for this $V[t]$ to decrease from its initial value to a fraction of $\ep$. On one hand, observe that the max/min operation ($p=\pm\infty$) only needs $d(\Gc)$ (the diameter of $\Gc$) iterations to project any point onto $\Cc$; on the other hand, for linear consensus scheme ($p=1$), it takes infinitely many iterations for $V[t]$ to go down to 0, as $V[t]$ is obtained by geometrically decreasing $\|\WB^t-\frac{1}{n}\1_n\1_n^\top\|_F^2$. It is then evident that, for small $\ep$ and large $p$, we have $T_{\ep,1}>T_{\ep,p}$, where $T_{\ep,p}$ is the $T_\ep$ for taking $p$-mean; that is, there exists a $p$-mean scheme converging faster than linear consensus in the pure consensus setting considered here.

For $\ep$ not small enough and the corresponding $T$ not large enough, this may not be true. In \pref{app:1}, we give an example where the one step decrease in sample variance of linear scheme is greater than that of max operation. However, we find from simulations that the ratio of sample variances given any $T$ monotonically decreases with $p$ ($p\geq1$) when $\xB[0]$ is uniformly distributed. This suggests that $p$-mean with a larger $p$ tends to project faster onto $\Cc$.

\subsection{Distributed Optimization with Nonlinear Transformation}\label{sec:3.2}
As we mentioned earlier, from the stochastic approximation viewpoint, many primal distributed optimization algorithms, such as NEXT and DGD, interleave the ``consensus steps" at a faster time-scale with the "descent steps" at a slower time-scale, where the consensus steps are subprojections onto the consensus plane $\Cc$ \cite{Kao_LocalApprox_2021,Borkar_Gossip_2016,Borkar_SABook_2008}. From \pref{sec:3.1}, we know that nonlinear subprojection by transformation is also a subprojection onto $\Cc$, which is often faster than the linear consensus scheme. Thus, it is of interest to study the convergence behavior of algorithms that combine the nonlinear subprojection with the original ``descent part" of a primal distributed optimization algorithm.

\subsubsection{NEXT with Nonlinear Transformation}\label{sec:3.2.1}
We consider taking element-wise transformations, where in each element the transformation function is strictly monotonic and bi-Lipschitz (and hence bi-continuous). The strict monotone property is to ensure the inverse transformation always exists.

\begin{dff}\label{dff:3.6}
Denote $\vf_{i,t}:\Kc\ra\Kc$ such that $\vf_{i,t}(\xB)=\begin{bmatrix}\vf_{i,t,1}(x_1)&\cdots&\vf_{i,t,d}(x_d)\end{bmatrix}^\top$ where $\xB=\begin{bmatrix}x_1&\cdots&x_d\end{bmatrix}^\top$ and $\{\vf_{i,t,l}\}_{l=1}^d$ are all one-dimensional functions. Note the three indices in the subscript of $\vf_{i,t,l}$ are the agent index, the time index, and the dimension index, respectively. Also, for any transformation $\vf:\Kc\ra\Kc$, define $\vf(S):=\{\vf(\xB):\xB\in S\}$.
\end{dff}

In the following revised version of NEXT, in each time step and for each agent we take an element-wise transformation of the intermediate result from the local optimization step, apply the doubly-stochastic matrix, and take the inverse element-wise transformation. It is easy to see that the combination of the three operations is a subprojection onto the consensus plane. Note that for different time steps and for different agents we could take different transformations, so that they could be chosen in an online manner depending on all the past information, giving us more flexibility.

\begin{algorithm}
\caption{NEXT with Nonlinear Transformation}\label{alg:3.1}
In Line 9 of \pref{alg:2.1}: {\footnotesize $\xB_i[t+1]=\vf_{i,t}^{-1}\left(\sum_{j=1}^nW_{ij}\vf_{i,t}(\zB_j[t])\right)$}
\end{algorithm}

We establish the same convergence result (i.e. to stationary points) as for NEXT \cite{Scutari_NEXT_2016} in the following theorem.

\begin{thm}\label{thm:3.7}
Let $\{\xB[t]\}_t\triangleq\{(\xB_i[t])_{i=1}^n\}_t$ be the sequence generated by \pref{alg:3.1}. Assume $\vf_{i,t,l}$ and $\vf_{i,t,l}^{-1}$ are strictly monotonic Lipschitz continuous functions with constants $L_+$ and $L_-$, respectively, for all $i\in[n]$, $t\in\{0\}\uU\N$ and $l\in[d]$. Under Assumption A and Assumption F in \pref{sec:2}, all sequences $\{\xB_i[t]\}_t$ asymptotically agree, and their limit points are stationary points of the original problem.
\end{thm}
\begin{proof}
See \pref{app:3}.
\end{proof}

Note that in \pref{alg:3.1}, node $j$ still sends $\zB_j[t]$ to node $i$; the transformation from $\zB_j[t]$ to $\vf_{i,t}(\zB_j[t])$ is taken  at node $i$.

\subsubsection{DGD with Nonlinear Transformation}\label{sec:3.2.2}
In the setting of NEXT, non-convex objectives and convergence to stationary points are considered; hence, it is hard to characterize the convergence rate in the setting. To analyze the effect of nonlinear transformations on convergence rate, we turn to another example that combines the transformations with DGD.

\begin{algorithm}
\caption{DGD with Nonlinear Transformation}\label{alg:3.2}
In Line 4 of \pref{alg:2.3}:\\
$\xB_i[t+1]=\vf_{i,t}^{-1}\left(\sum_{j=1}^nW_{ij}\vf_{i,t}(\xB_j[t]-\al[t]\gd f_j(\xB_j[t])\right)$
\end{algorithm}

\begin{thm}\label{thm:3.8}
Assume $\vf_{i,t,l}$ and $\vf_{i,t,l}^{-1}$ are strictly monotonic Lipschitz continuous functions and have uniformly bounded first derivatives and second derivatives\footnote{We actually only need the existence of the second derivatives of $\vf_{i,t,l}$ and $\vf_{i,t,l}^{-1}$; the boundedness of $\xB_i[t]$ will then ensure the boundedness of the derivatives in the region of interest. We use the same constants for simplicity of notation.}, all with constants $L_+$ and $L_-$, respectively, for all $i\in[n]$, $t\in\{0\}\uU\N$ and $l\in[d]$. Also for all $i\in[n]$, assume the objective functions $f_i$ is strongly convex with constant $\nu$, has gradient uniformly bounded by $B$ (Assumption A3), and has Lipschitz continuous gradient with constant $L_f$ (Assumption A4 with $L_f=\max_iL_i$), which implies the existence of a unique optimal solution $\xB^*$. Suppose we choose a constant learning rate satisfying $\al[t]=\al<\frac{1}{L_f}$. Then the sequence generated by \pref{alg:3.2} $\{\xB[t]\}_t\triangleq\{(\xB_i[t])_{i=1}^n\}_t$ satisfies
\bal\label{eq:3.14}
\frac{n\nu}{2}\|\xBb[t]-\xB^*\|^2&\leq F(\xBb[t])-F(\xB^*)\\
&\leq\bar{\rho}^t\big[F(\xBb[0])-F(\xB^*)\big]+O(\al^2),
\eal
where the factor $\bar{\rho}\triangleq1-\al\nu$ is in $[0,1)$, and $\xBb\triangleq\frac{1}{n}\sum_{i=1}^n\xB_i$ is the average vector. In other words, $\xBb[t]$ converges to an $O(\al)$ neighborhood of $\xB^*$ exponentially fast, while the optimality gap in the mean $F(\xBb[t])-F(\xB^*)$ also decreases exponentially until it settles to a value that is $O(\al^2)$. For the agreement from the nodes, for all $i\in[n]$, the deviation from the mean $\|\xB_i[t]-\xBb[t]\|$ yet again decreases exponentially until reaching $O(\al)$ (see \pref{lem:B.3} for the detailed description).
\end{thm}
\begin{proof}
See \pref{app:4}.
\end{proof}

In \pref{thm:3.8} and its proof, we see that the convergence rate is in line with the results in the literature \cite{Yuan_DGDConv_2016,Wei_MultCons_2018}. That is, taking nonlinear transformations does not affect the factor $\bar{\rho}$ in the geometric convergence, but leads to larger constants in the expressions of the $O(\al)$ and $O(\al^2)$ neighborhoods. Hence, based on our proof methodology, the nonlinear consensus scheme does not improve the theoretical guarantee of the convergence rate. However, we do see improvements from using various nonlinear consensus schemes in \pref{sec:6}. Some of them arise from the detailed relative relations for the variables exemplified in \pref{sec:5.2} that are hard to be captured by current analysis methods.

\section{Relaxation of Column Stochasticity for Consensus}\label{sec:4}
Many primal distributed optimization algorithms require doubly stochastic matrices for averaging. In contrast, in this section we study relaxing the need of column stochasticity. The NEXT algorithm (as well as many primal algorithms that use the averaging consensus scheme) is essentially performing a two time-scale stochastic approximation \cite{Kao_LocalApprox_2021}, where the algorithm interlaces the descent steps and the consensus steps. In the fast process consisting of the consensus steps, the decision variables maintained in the nodes asymptotically agree as the ensemble vector $\XB=[\xB_1^\top\es\cdots\es\xB_n^\top]^\top$ converges to the consensus plane $\Cc$, which is referred to as ``\emph{consensus convergence}." On the other hand, in the slow process consisting of the descent steps, the consensus vector $\xB_c$ (which is suitably taken as the average vector $\xBb=\frac{1}{n}\sum_{i=1}^n\xB_i$) converges to one of the local minima, which we call ``\emph{aggregate convergence}." Row stochasticity is required for both of these convergences, as it ensures the new iterate stays inside the convex hull spanned by the old iterates during averaging; otherwise, the new iterate could fall outside the feasible region, let alone guaranteeing convergence. On the other hand, column stochasticity is to ensure that the objectives of the nodes weigh in equally, since the overall objective is the sum of them (and hence equivalent to the uniformly weighted average of them); hence, it is necessary for aggregate convergence but not consensus convergence. Here, we split out the two convergences, and relax the column stochasticity requirement for the consensus convergence which leads to more general and flexible consensus schemes.

\subsection{NEXT with Column Stochasticity of Consensus Relaxed}\label{sec:4.1}
We start with the definition of a ``shrunk convex hull."

\begin{dff}\label{dff:3.9}
Let $co(T)$ be the convex hull of a finite set $T$. Also, let $m_S$ be the centroid of a convex set $S$, and $\dl\circ S:=\{(1-\dl)m_S+\dl x:x\in S\}$, i.e. the set obtained by shrinking $S$ towards its centroid by a factor of $\dl\in[0,1]$. The shrunk convex hull of $T$ with factor $\dl$ is then defined as $\dl\circ co(T)$.
\end{dff}

The inexact NEXT algorithm can be revised by choosing any point inside the shrunk convex hull spanned by the neighboring local optimization results with a factor $\dl$ smaller than $1$, as given in the following algorithm.

\begin{algorithm}
\caption{NEXT with Consensus in Convex Hull}\label{alg:3.3}
In Line 9 of \pref{alg:2.1}:\\
pick any $\xB_i[n+1]\in\dl\circ co(\{\zB_j[t]:j\in Nb(i)\})$
\end{algorithm}

\begin{thm}\label{thm:3.10}
Let $\{\xB[t]\}_t\triangleq\{(\xB_i[t])_{i=1}^n\}_t$ be the sequence generated by \pref{alg:3.3}. Assume $\dl\in[0,1)$. Under Assumption A and Assumption F, all sequences $\{\xB_i[t]\}_t$ asymptotically agree, and their limit points are stationary points of the original problem.
\end{thm}
\begin{proof}
See \pref{app:5}.
\end{proof}

In short, the above theorem says that the NEXT algorithm can work well with $\xB$ and $\yB$ using different consensus matrices, and the consensus matrices for $\xB$ do not have to be column stochastic. Column stochasticity is, however, crucial for the consensus matrices for $\yB$, since it ensures $\yB$ tracks the uniformly weighted average of $\gd f_i$'s, and hence $\xB$ converges to the correct optima of $\sum_if_i$ instead of a non-uniformly weighted version.

\subsection{DGD with Gradient Tracking and Column Stochasticity of Consensus Relaxed}\label{sec:4.2}
In the DGD algorithm, on the other hand, the gradient information is gossiped through the averaging of $\xB$. Therefore, relaxing column stochasticity of the averaging matrices could lead to convergence to the optima of a non-uniformly weighted version of the objective functions $f_i$'s. One approach that allows DGD to enjoy the flexibility of column stochasticity relaxation is through tracking the gradient information with another variable $\yB$ \cite{Notarstefano_DistSurvey_2019} given in \pref{alg:3.4}, which is still averaged by doubly stochastic matrices, while $\xB$ is averaged by possibly non-column-stochastic matrices.

\setcounter{AlgoLine}{0}
\begin{algorithm}
\caption{DGD with Gradient Tracking}\label{alg:3.4}
\nl Initialization: $\xB_i[0]\in\R^d$, $\yB_i[0]=\gd f_i(\xB_i[0])$, $t=0$.\\
\nl \While{$\xB[t]$ does not satisfy the termination criterion}{
\nl $t\la t+1$\\
\nl $\xB_i[t+1]=\sum_jW_{ij}^x\xB_j[t]-\al[t]\yB_i[t]$\\
\nl $\yB_i[t+1]=\sum_jW_{ij}^y\yB_j[t]+\gd f_i(\xB_i[t+1])-\gd f_i(\xB_i[t])$
}
\textbf{Output:} $\xB[t]$
\end{algorithm}

To ensure convergence to correct optima of $\sum_if_i$, the matrix $\WB^y$ has to be doubly stochastic while the matrix $\WB^x$ only has to be row stochastic. The detailed convergence analysis of this algorithm is left as future work.

\section{Combining Nonlinear Transformation and Column Stochasticity Relaxation}\label{sec:5}
Recall that in \pref{sec:3.1} we argued that the max and min operations are generally faster than linear consensus schemes -- it only takes the diameter of $\Gc$ steps for the operations to reach exact consensus. A natural question then is how will primal distributed optimization algorithms perform when the consensus is reached through these operations, and can we establish convergence guarantees for such schemes. For the former question, we do find in some numerical examples in \pref{sec:6} that algorithms with max and min operations used for consensus converge faster than those with linear consensus schemes. We are hence motivated to study the latter question by investigating the nonlinear transformations of the shrunk convex hull, a combination of the ideas from \pref{sec:3.2} and \pref{sec:4}, with NEXT algorithm as the example.

\begin{algorithm}
\caption{NEXT with Nonlinear Transformation and Column Stochasticity Relaxation}\label{alg:3.5}
In Line 9 of \pref{alg:2.1}:\\
pick any $\xB_i[t+1]\in\vf_{i,t}^{-1}(\dl\circ co(\{\vf_{i,t}(\zB_j[t]):j\in Nb(i)\}))$
\end{algorithm}

\begin{thm}\label{thm:3.11}
Let $\{\xB[t]\}_t\triangleq\{(\xB_i[t])_{i=1}^n\}_t$ be the sequence generated by \pref{alg:3.5}. Assume $\vf_{i,t,l}$ and $\vf_{i,t,l}^{-1}$ are strictly monotonic Lipschitz continuous functions with constants $L_+$ and $L_-$, respectively, for all $i\in[n]$, $t\in\{0\}\uU\N$ and $l\in[d]$. Also assume $\dl\in[0,1)$, Assumption A, and Assumption F. Then all sequences $\{\xB_i[t]\}_t$ asymptotically agree, and their limit points are stationary points of the original problem.
\end{thm}
\begin{proof}
The result follows by combining the proofs of \pref{thm:3.7} and \pref{thm:3.11}. In fact, the proof of \pref{thm:3.7} essentially applies here too since we do not require the $\WB$ matrices to be column stochastic nor time-invariant in the proof. The only requirements regarding $\WB$ are the assumptions in \pref{sec:2.1}, which is ensured by $\dl<1$.
\end{proof}

\subsection{Enlarging to Cube Hull with $p$-means Transformations}\label{sec:5.1}
One reason that the max operation sometimes outperforms linear consensus is it has larger ``step size" in the consensus step. Indeed, it is common that the outcome of the max operation falls outside the convex hull spanned by the neighboring iterates. In general, performing the max operation can lead to local variables falling outside $\Kc$; however, with additional constraint on $\Kc$, e.g. $\Kc=\R_+^d$, we can further expand the ``hull of consensus choices" to the \emph{cube hull} defined below using the family of $p$-means transformations to include the max and min operations. The main idea is to take the union of all the convex hulls transformed by any transformation inside the family. The theory developed here serves as an application and hence a special case of \pref{thm:3.11}.

\begin{dff}\label{dff:3.12}
The cube hull $cb(T)$ of a finite set $T\rU\R^d$, is defined as the smallest $d$-dimensional cube that contains $T$ with all edges parallel to the Cartesian axes. That is, if we write $T=\{\xB_1,\dots,\xB_n\}$, then $cb(T):=\{\yB:\min_ix_{i,l}\leq y_l\leq\max_ix_{i,l}\es\forall\es i\in[n]\text{ and }l\in[d]\}$.
\end{dff}

Recall that $\vf_{i,t}$ can be chosen at time $t$ dependent on all the past information. Suppose we have a family of transformations $\Phi$. \pref{alg:3.5} then implies that we can choose $\xB_i[t+1]$ from an even bigger set $\buU_{\vf\in\Phi}\vf^{-1}(\dl\circ co(\{\vf(\zB_j[t]):j\in Nb(i)\}))$. Let us take $\Phi$ to be the family of power of $p$ functions $\{\hat{\vf}_p:p\in\R\sm\{0\}\}$ where $\hat{\vf}_p(x)=x^p$ with set $\Kc=\R_+^d$.

\begin{fact}\label{fact:3.13}
Given any finite set $T\in\R_+^d$ and $\dl'\in[0,1)$, there exist a $\pb<\infty$ and $\dl\in(0,1)$ such that $\dl'\circ cb(T)\rUe\buU_{\vf\in\Phi_{\pb}}\vf^{-1}(\dl\circ co(\{\vf(T)\}))$, where $\Phi_{\pb}$ is defined as the family of power functions $\{\hat{\vf}_p:p\in[-\pb,\pb]\sm\{0\}\}$.
\end{fact}

\begin{figure}
\centering
\subfloat[$p_1=10,p_2=10$]{
\includegraphics[scale=0.4]{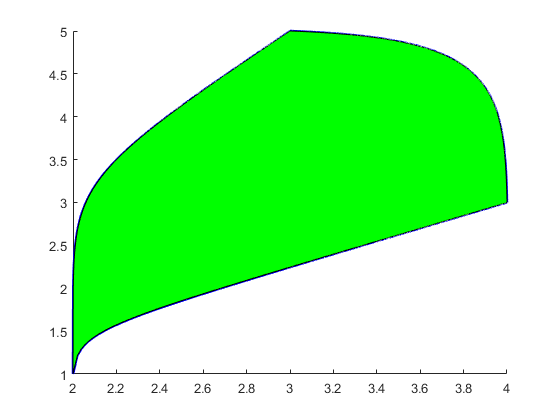}}
\subfloat[$p_1=10,p_2=-10$]{
\includegraphics[scale=0.4]{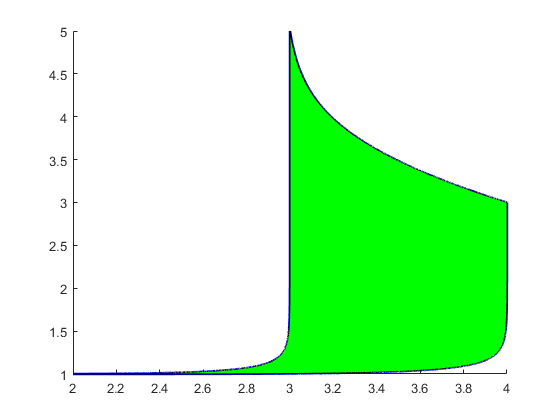}}\\
\subfloat[$p_1=-10,p_2=10$]{
\includegraphics[scale=0.4]{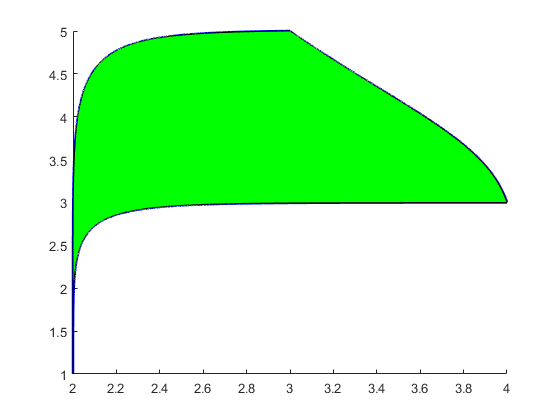}}
\subfloat[$p_1=-10,p_2=-10$]{
\includegraphics[scale=0.4]{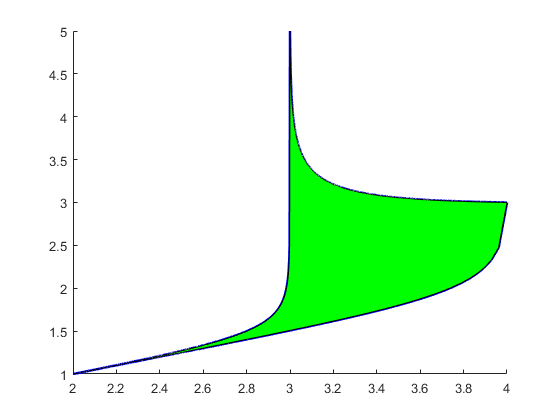}}
\caption{An example of transformed convex hull with two-dimensional element-wise $p$-power functions $\vf_{p_1,p_2}(\xB)=[\hat{\vf}_{p_1}(x_1)\es\hat{\vf}_{p_2}(x_2)]^\top$.}
\label{fig:3.1}
\end{figure}

\pref{fig:3.1} gives an example of ``transformed convex hull" $\vf^{-1}(co(\{\vf(T)\}))$ using the family of power functions. In the example, we let $T=\{[2\es1]^\top,[3\es5]^\top,[4\es3]^\top\}$, and $\vf_{p_1,p_2}(\xB)=[\hat{\vf}_{p_1}(x_1)\es\hat{\vf}_{p_2}(x_2)]^\top$. The four sub-figures depict the region of the transformed hull with four different values of $(p_1,p_2)$. To get a better sense of their union, we plot them together in \pref{fig:3.2}. From the figure, we can indeed see that $0.8\circ cb(T)$ is contained in the union of the transformed hulls with $\vf$ coming from $\{\vf_{p_1,p_2}:p_1,p_2\in[-10,10]\sm\{0\}\}$\footnote{Note that the blank region in the bottom left corner in \pref{fig:3.2} can be filled by slowly changing $(p_1,p_2)$ from $(1,1)$ to $(10,-10)$.}.

\begin{figure}
\centering
\includegraphics[scale=0.5]{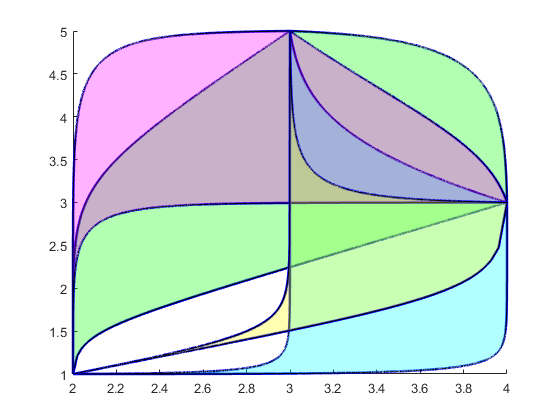}
\caption{The four regions in \pref{fig:3.1} plotted in different colors in one graph.}
\label{fig:3.2}
\end{figure}

By \pref{fact:3.13}, the convergence of \pref{alg:3.5}, and the fact that all functions in $\Phi_{\pb}$ as well as their inverses are Lipschitz continuous, we can choose any $\xB_i[t+1]$ from the ``shrunk cube hull" supported by neighbors' variables, given in \pref{alg:3.6}. Its convergence is then ensured by \pref{thm:3.11}.

\begin{algorithm}
\caption{NEXT with Consensus in Cube Hull}\label{alg:3.6}
In Line 9 of \pref{alg:2.1}:\\
pick any $\xB_i[n+1]\in\dl\circ cb(\{\zB_j[t]:j\in Nb(i)\})$
\end{algorithm}

\vspace{-10pt}
\subsection{Application: Gradient-Oriented Consensus Schemes}\label{sec:5.2}
In the numerical simulations presented in \pref{sec:6}, NEXT with max and min consensus schemes can converge either faster or slower than the linear counterpart depending on the initial values $\{\xB_i[0]\}_{i\in[n]}$. When the initial values are generally component-wise smaller (resp. larger) than the optimal value, max (resp. min) works better and min (resp. max) works poorly. When the initial values are around the same range as the optimal value, with some coordinates larger and some smaller, then usually linear consensus is better than max and min versions of nonlinear consensus. Since max and min consensus schemes are generally faster subprojections onto the consensus plane, the described phenomena also involve the aggregate convergence part. Recall that the iterates generally descend in directions \textit{similar} to the negative gradient for the part. When the initial values are smaller (resp. larger) than the optimal value, the direction of change of taking max (resp. min) operation is more similar to the negative gradient, and thus it further reduces the objective function. This gives rise to the idea of the ``\textit{gradient-oriented consensus scheme}," where we try to align the consensus steps with the negative gradients within the hull of consensus choices.

We start by considering the hull of consensus choices being the convex hull spanned by the neighboring iterates as in \pref{alg:3.3}. The goal is to choose a point such that the direction of change is similar to the negative \emph{total} gradient. In the distributed setting, a node does not have the information of objectives of other nodes and hence does not know total gradient. Fortunately, NEXT uses the $\yB$ variable to track the average gradient, which we take advantage of in \pref{alg:3.7}.

\begin{algorithm}
\caption{NEXT with Gradient-oriented Consensus in Convex Hull}\label{alg:3.7}
In Line 9 of \pref{alg:2.1}:\\
$\xB_i[t+1]\in\argmax_{\xB\in\dl\circ co(\{\zB_j[t]:j\in Nb(i)\})}\frac{\yB_i[t]^\top(\xB-\zB_i[t])}{\|\xB-\zB_i[t]\|}$
\end{algorithm}

\noindent Basically, \pref{alg:3.7} is doing angle minimization within the convex hull supported by the variables from the neighbors. It is a special case of \pref{alg:3.3} and uses the gradient information. The idea of gradient angle minimization was proposed for gradient descent methods in constrained settings \cite{Zoutendijk_Methods_1960}. One could also consider projecting the gradient or optimizing any other objective function using gradient information within the constraint set $\dl\circ co(\{\zB_j[t]:j\in Nb(i)\})$.

As we enlarge the hull of consensus choices to the shrunk cube hull, we can adopt the same gradient-oriented idea but with the new iterate lying in the constraint set $\dl\circ cb(\{\zB_j[t]:j\in Nb(i)\})$ as in \pref{alg:3.8}.

\begin{algorithm}
\caption{NEXT with Gradient-oriented Consensus in Cube Hull}\label{alg:3.8}
In Line 9 of \pref{alg:2.1}:\\
$\xB_i[t+1]\in\argmax_{\xB\in\dl\circ cb(\{\zB_j[t]:j\in Nb(i)\})}\frac{\yB_i[t]^\top(\xB-\zB_i[t])}{\|\xB-\zB_i[t]\|}$
\end{algorithm}

\noindent Just as \pref{alg:3.7} is a special case of \pref{alg:3.3}, \pref{alg:3.8} is a special case of \pref{alg:3.6} (which is in turn a special case of \pref{alg:3.5}). The convergence of the algorithms are guaranteed by \pref{thm:3.10} and \pref{thm:3.11}. We remark that the cube hull is the largest possible set one could consider, as the notion of mean generally requires that the average lies between min and max \cite{Bauso_Cons_2006}. Finally, our conjecture is the gradient-oriented idea within shrunk cube hull or shrunk convex hull also works for DGD with gradient tracking as explained in \pref{sec:4.2}; the corresponding convergence guarantee and performance evaluation are left as future work.

\section{Simulation Results}\label{sec:6}
In this section, we simulate the convergence of NEXT with numerous consensus schemes proposed in this paper. We assume the underlying graph is the 19 cell wrap-around implementation \cite{Huo_Wrap_2005}, a classic example widely used in simulations for cellular networks. Each BS in the implementation is a node in our graph, and an edge exists between a pair of nodes if and only if the corresponding BSs are neighbors. The underlying graph is a symmetric regular graph with 19 nodes, each with degree 6.

We consider the objective functions having a ``partial dependency structure." This is one scenario when the ordinary linear consensus might not work well. In particular, each node $i$ corresponds to a local variable $\xB^i\in\R_+^{d'}$, and its objective $f_i$ only depends on its neighbors' as well as its own local variables. That is, $f_i$ only depends on $\xB^{Nb(i)}$, the ensemble vector with elements in $\{\xB^j:j\in Nb(i)\}$. We let $d'=2$, so the whole tuple variable $\xB$ lies in $\R_+^{38}$. For all $i$, we consider a convex quadratic objective, where $f_i$ is given as
\beq\label{eq:3.15}
f_i(\xB^{Nb(i)})=\frac{1}{2}{\xB^{Nb(i)}}^\top\MB^\top\MB\xB^{Nb(i)}+\bB^\top\xB^{Nb(i)},
\eeq
$\MB$ is a $14\times14$ matrix with \textit{i.i.d.} elements from $\text{Unif}[-1,1]$, and $\bB$ is a $14\times1$ vector with \textit{i.i.d.} elements from $\text{Unif}[-150,-50]$. The negative choices of $\bB$ is to make the optimal solution in $\R^d$ typically lie in $\R_+^d$. The mean of the optimal solution for each coordinate is around 25. The initial choice of $\xB[0]$ for every node is drawn \textit{i.i.d.} for all coordinates from $\chi^2(k)$ (chi-squared distribution with parameter $k$), where we choose $k=5$, $k=25$, and $k=100$ for our three cases, corresponding to low, median, and high initial values relative to the optimal value. The surrogate function of $f_i$ is formed by direct linearization plus a quadratic regularization term with coefficient $\tau_i=100$ for all $i$. Also let $\al[t]=0.8\cdot(t+1)^{-0.53}$, and $\dl=0.9$.

\begin{figure*}
\centering
\subfloat[$k=5$]{
\includegraphics[scale=0.33]{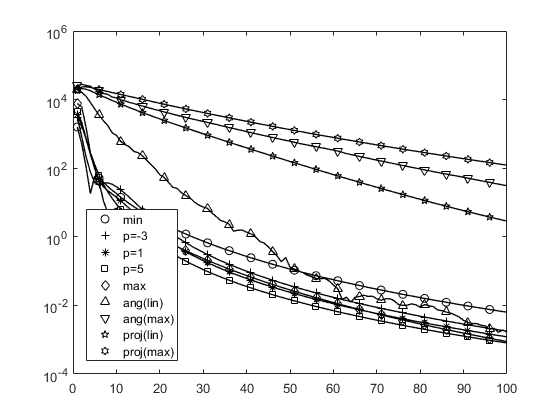}}
\label{fig:sim1a}
\subfloat[$k=25$]{
\includegraphics[scale=0.33]{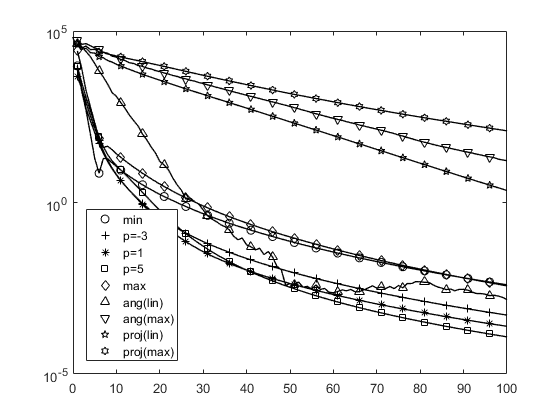}}
\label{fig:sim1b}
\subfloat[$k=100$]{
\includegraphics[scale=0.33]{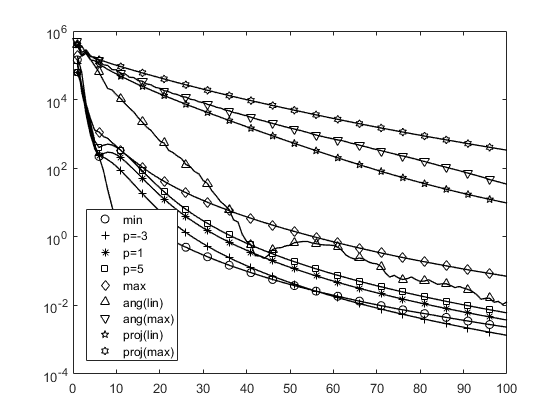}}
\label{fig:sim1c}
\caption{The convergence rates of the objective value with respect to number of iterations for different consensus methods for randomly generated quadratic objective functions.}
\label{fig:3.3}
\end{figure*}

\begin{figure*}
\centering
\subfloat[$k=5$]{
\includegraphics[scale=0.33]{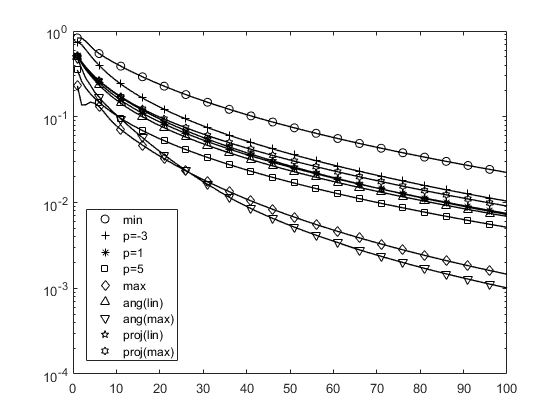}}
\label{fig:sim2a}
\subfloat[$k=25$]{
\includegraphics[scale=0.33]{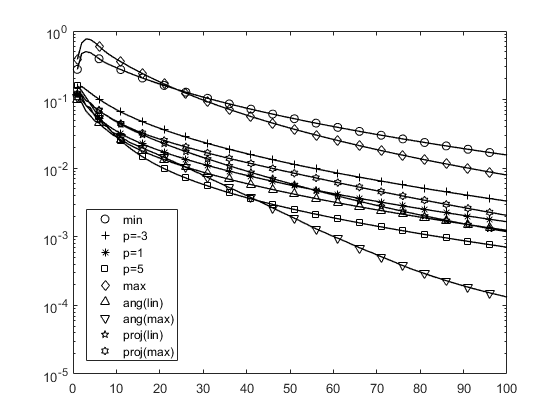}}
\label{fig:sim2b}
\subfloat[$k=100$]{
\includegraphics[scale=0.33]{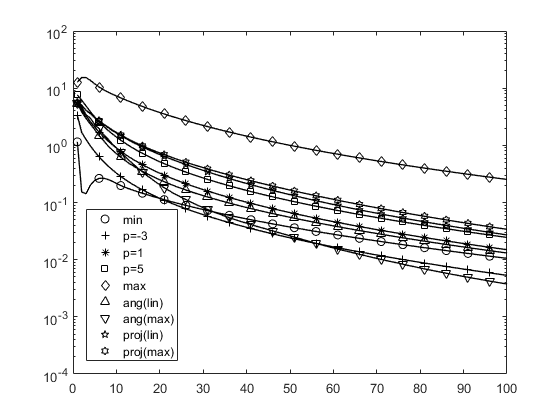}}
\label{fig:sim2c}
\caption{The convergence rates of the total deviation from the optimum with respect to number of iterations for different consensus methods for randomly generated quadratic objective functions.}
\label{fig:3.4}
\end{figure*}

Denote $\xB_j^i$ to be the variable part of $\xB^i$ stored at node $j$, and let $\xB_a:=[\cdots\es{\xB_i^i}^\top\es\cdots]^\top$. Note that the unique optimal value $F^*=F(\xB^*)$ achieved at $\xB^*$ can be easily solved by quadratic programming. \pref{fig:3.3} illustrates how $\frac{|F(\xB_a)-F^*|}{|F^*|}$ decreases as a function of number of iterations, and \pref{fig:3.4} shows the variation of $\sum_{i=1}^n\|\xB_i^i-\xB^*\|^2$. The plotted consensus schemes include the following: (1) component-wise $(p,w)$-mean scheme, i.e. \pref{alg:3.1} with $\vf_{i,t,l}=\hat{\vf}_p$ for all $i,t,l$ and taking linear average with weights $\WB$ given by
\[
W_{ij}=\left[\I\{i=j\}(\bar{d}-d_i)+\I\{i\neq j\}\right]/\bar{d},
\]
for $j\in Nb(i)$ where $\bar{d}=\max_id_i+1$ (and obviously $W_{ij}=0$ for $j\not\in Nb(i)$), (2) component-wise max and min operations, which are just $(p,w)$-mean and taking $p\ra\infty$ and $p\ra-\infty$\footnote{Note that our theory developed in \pref{sec:5} does not guarantee the convergences of max and min operations but a shrunk version of them.}, (3) angle minimization within shrunk convex hull given in \pref{alg:3.7}, and (4) angle minimization within shrunk cube hull given in \pref{alg:3.8}.

Generally, when the initial value is component-wise lower (resp. higher) than optimal value in the case of $k=5$ (resp. $k=100$), the optimal $p$ for fastest convergence is large (resp. small); when the initial value is around optimal value ($k=25$), the optimal $p$ is around 1. Let us focus on the $(p,w)$-mean first. \pref{fig:3.3} shows that when $k=5$, choosing $p=5$ makes convergence the fastest while max operation is the second best and min is the worst; when $k=25$, $p=5$ is still the fastest but $p=1$ is the second best, whereas max and min are the worst; when $k=100$, $p=-3$, min, and max are the best, second best, and worst. The gradient-oriented consensus in convex hull method decreases slower at first, but later catches up and lies somewhere between the group of $(p,w)$-mean schemes in all cases. Note that the convergence curves of the gradient-oriented methods can be adjusted by $\dl$ -- the smaller $\dl$ is, the closer the curves are to the linear method ($p=1$). The reason that min does not work well in \pref{fig:3.3} (a) is as follows. At each iteration, at node $i$ the component $\xB_i^j$ for $j\not\in Nb(i)$ is dragged upward by its neighbors towards the optimum, and the min operation drags the value back. Therefore, even if min is the fastest subprojection similar to max, it does not work well in this scenario. For the same reason max does not work well in \pref{fig:3.3} (c).

In \pref{fig:3.4} and for the $(p,w)$-mean family, max is the best and min is the worst for low initial value; $p=5$ is the best and max and min are the worst for median initial value; $p=-3$ is the best followed by min while max is the worst for high initial value. The angle minimization within convex hull algorithm outperforms the linear scheme slightly in all cases, but does not surpass the best method of the $(p,w)$-mean family in each case. The cube hull angle minimization method descends the fastest, as it tries to descend in the negative gradient direction as much as possible within the largest possible choice set, leading to the fastest decrease in objective value.

The main contribution of proposing these consensus schemes is providing the flexibility for the algorithm so that the ``best choice," in terms of the number of iterations required, is available for the given requirements. When running distributed optimization algorithms, there are two common termination criteria: one regarding the convergence of $\xB_i$ to $\xB^*$ for all $i$, and another concerning the convergence of $F(\xBb)$ to $F(\xB^*)$. If one cares about the latter much more than the former, e.g. the goal is to find the optimal value $F(\xB^*)$ within a tolerance of $10^{-3}$, then gradient-oriented within cube hull is likely the best choice from the simulation. If one cares about the former, from \pref{fig:3.3} we know that usually there is a better scheme in the $(p,w)$-mean family than the linear consensus. If the application focuses on a part of the objective function rather than the whole, we can also revise the objective function in \pref{alg:3.8} by tracking a gradient of that part.

\section{Conclusion and Future Directions}\label{sec:7}
In this paper, we studied various nonlinear consensus schemes for distributed optimization. We pointed out that from the stochastic approximation viewpoint, the consensus step is a subprojection onto the consensus plane which is the fast process, while the consensus variable descends to the optimum on the plane which is the slow process. From this perspective, many subprojections outside the paradigm of linear consensus can be considered for distributed optimization. We considered taking monotonic nonlinear transformations before taking linear averages; we established convergence when such consensus schemes are combined with NEXT, and analyzed the convergence rate when it is combined with DGD. We further reestablished the convergence result for NEXT when the averaging matrices are no longer column stochastic. Combining this relaxation with the nonlinear transformation idea allows us to choose any point in the ``shrunk cube hull" as the next iterate during the consensus step, which is a general consensus scheme for distributed optimization. Numerical results show that depending on the relation between initial variables and the optimum, various proposed schemes can outperform traditional linear consensus schemes.

One future direction is to investigate the local properties between iterates, objective functions, and consensus schemes, as the numerical advantages of the proposed consensus schemes seem to be data dependent. This dependency prohibits us to manifest the convergence rate advantages through the worst case analysis done in \pref{thm:3.8}. Another interesting direction would be extending the element-wise nonlinear transformations to \emph{Riemannian manifolds} \cite{Lee_Manifolds_2018}, which are invertible and preserve local geometric relations. In this paper, we saw that convergence can be ensured with Lipschitz element-wise transformations. It is of interest to study the cases where the coordinates are transformed through Riemannian manifolds altogether instead of individually, and identifying the conditions of the transformations that preserve convergence.

\bibliographystyle{IEEEtranS}
\bibliography{References}

\newpage

\appendix
\appendixpage

\section{Discussion for $p$-mean in Pure Consensus Setting}\label{app:1}
We first formally define $T_\ep$, and give the statement regarding $T_{\ep,p}$ in the following fact.
\beq\label{eq:3.13}
T_\ep\triangleq\inf\left\{t:\frac{V[t]}{V[0]}=\frac{\sum_i(x_i[t]-x^*)^2}{\sum_i(x_i[0]-x^*)^2}\leq\ep\right\}.
\eeq

\begin{fact}\label{fact:3.4}
Denote $T_{\ep,p}$ as the minimal number of iterations required for the sample variance to reduce to a fraction of $\ep$ using the $p$-mean consensus scheme in \eqref{eq:3.12}. Then given any $\xB[0]$, there exists small enough $\ep>0$ and large enough $p>1$ such that $T_{\ep,1}>T_{\ep,p}$.
\end{fact}

Given any consensus scheme, achieving a smaller \textit{tolerance} $\ep$ requires a larger $T_\ep$. For a fixed tolerance $\ep$, a faster consensus scheme takes smaller $T$; and given $T$, a faster consensus scheme achieves smaller $\ep$. The above fact states that for smaller tolerance requirement or larger number of iterations running, there exists a $p$-mean scheme converging faster than linear consensus in the pure consensus setting considered here. This is, however, not necessarily true when $T$ is not large enough, as given in the following counterexample.

\begin{ex}\label{ex:3.5}
Consider a ring of five nodes $\Vc=\{1,2,3,4,5\}$ and $\Ec=\{(1,2),(2,3),(3,4),(4,5),(5,1)\}$ with initial values $\xB[0]=[7\es2\es12\es2\es7]^\top$ and a doubly stochastic matrix $\WB$ where self-loops are weighed by $\frac{3}{5}$ and other edges are weighed by $\frac{1}{5}$. Then after one iteration, linear consensus applies $\WB$ on $\xB[0]$ and gives $[6\es5\es8\es5\es6]^\top$, while max operation yields $[7\es12\es12\es12\es7]^\top$. The ratio of sample variances of linear scheme $\frac{3}{35}$ is smaller than that of max operation $\frac{1}{5}$.
\end{ex}

\section{The Shrinking Span Lemma}\label{app:2}

\begin{dff}\label{dff:B.1}
For an ensemble vector $\XB\triangleq[\xB_1^\top\es\cdots\es \xB_n^\top]^\top$ where $\xB_i=[x_{i,1}\es\cdots\es x_{i,d}]^\top\in\R^d$ for $i\in[n]$, we define its range $sp(\XB)$ to be
\[
sp(\XB)\triangleq d\cdot\max_{l\in[d]}\left(\max_ix_{i,l}-\min_ix_{i,l}\right).
\]
Note that if we let $\XB_l\triangleq[x_{1,l}\es\cdots\es x_{n,l}]^\top$ to be the dimension $l\in[d]$ of the ensemble, then the range can be expressed by $sp(\XB)=d\cdot\max_{l\in[d]}sp(\XB_l)$. In the context of this paper, the $i\in[n]$ is the agent index while the $l\in[d]$ is the dimension index of the decision variable.
\end{dff}

\begin{lem}\label{lem:B.2}
Let $\WB\in\R^{n\times n}$ be a row stochastic matrix with all elements $w_{ij}\in[\vth,1-\vth]$ for some $\vth\in(0,0.5)$, and $\vB\in\R^n$ be a vector. Then $sp(\WB\vB)\leq(1-2\vth)sp(\vB)$.
\end{lem}
\begin{proof}
Let $a\in\argmax_iv_i$, $b\in\argmin_iv_i$, $a'\in\argmax_i(Wv)_i$, and $b'\in\argmin_i(Wv)_i$. Then $sp(\vB)=v_a-v_b$, and
\allowdisplaybreaks[4]
\begin{align*}
sp(\WB\vB)
&=\left(\sum_{j=1}^nW_{a'j}v_j\right)-\left(\sum_{j=1}^nW_{b'j}v_j\right)\\
&=W_{a'b}v_b+\left(\sum_{j\neq b}W_{a'j}v_j\right)-W_{b'a}v_a-\left(\sum_{j\neq a}W_{b'j}v_j\right)\\
&\leq W_{a'b}v_b+\left(\sum_{j\neq b}W_{a'j}\right)v_a-W_{b'a}v_a-\left(\sum_{j\neq a}W_{b'j}\right)v_b\\
&=W_{a'b}v_b+\left(1-W_{a'b}\right)v_a-W_{b'a}v_a-\left(1-W_{b'a}\right)v_b\\
&=\left(1-W_{a'b}-W_{b'a}\right)\cdot(v_a-v_b)\leq(1-2\vth)sp(\vB).
\end{align*}
\end{proof}

\section{Proof of \pref{thm:3.7}}\label{app:3}\hypertarget{pf:thm:3.5}{}
The main intuition behind the proof is that even with time-varying and agent-dependent transformations, the consensus scheme is still a subprojection converging exponentially fast onto the consensus plane $\Cc$. \pref{lem:B.2} is the prototype of the convergence, which shows that the range of a vector shrinks geometrically when being multiplied by a doubly stochastic matrix with positive elements. With a fixed transformation, it is easy to see that the convergence property remains -- performing a change of variable $z=\vf(x)$, the nonlinear consensus scheme is a linear consensus scheme in the $z$-domain, so the range of $z$ decreases exponentially, which also implies the range of $x$ decreases exponentially because of the Lipschitz assumption. However, when the communication graph is not fully connected, the transformation is time-varying and agent-dependent, and the process is coupled with the distributed descent process, more careful treatment is required.

We start by showing that the nonlinear consensus scheme with time-varying and agent-dependent transformations and a fixed (possibly not fully-) connected communication graph $\Gc=(\Nc,\Ec)$ is a subprojection onto $\Cc$.  Let $D=diam(\Gc)$ be the diameter of the graph $\Gc$. It is easy to show that for any doubly stochastic matrix $\WB$ satisfying the constraint given in \pref{sec:2.1} ($W_{ij}>0$ if and only if $(i,j)\in\Ec$ for $i\neq j$ and $W_{ii}>0$), the matrix $\WB^D$ is a doubly stochastic matrix with positive elements. Let $\vth>0$ be the smallest entry of $\WB$. Without the transformations, the range of the vector will shrink by $(1-2\vth^D)$ from \pref{lem:B.2} every $D$ steps as the smallest entry in $\WB^D$ is $\vth^D$. We show that this is still the case with the transformations. We consider the averaging process
\[
\xB_i[t+1]=\vf_{i,t}^{-1}\left(\sum_{j=1}^nW_{ij}\vf_{i,t}(\xB_j[t])\right);
\]
note that without the descent process, different dimensions do not couple together, so we only focus on dimension $l\in[d]$: $x_{i,l}[t+1]=\vf_{i,t,l}^{-1}\left(\sum_{j=1}^nW_{ij}\vf_{i,t,l}(x_{j,l}[t])\right)$ and from now on we will omit the dimension index $l$. Consider at time $t$, we have $x_{\max}\triangleq\max_ix_i[t]$ and $x_{\min}\triangleq\min_ix_i[t]$, so that $sp(\XB[t])=x_{\max}-x_{\min}$ where $\XB=[x_1\es\cdots\es x_n]^\top$ is the ensemble vector of the variables from all agents (only for the considered dimension $l$). Let $i^*\in\argmax_ix_i[t]$ be one of the nodes with the maximum value at $t$; the value will spread through the network within $D$ steps. For any node $j$ that is a direct neighbor of $i^*$, assuming without loss of generality that $\vf_{j,t}$ is increasing instead of decreasing, we have
\allowdisplaybreaks[4]
\begin{align*}
\vf_{j,t}(x_j[t+1])&=\sum_kW_{jk}\vf_{j,t}(x_k[t])
=W_{ji^*}\vf_{j,t}(x_{i^*}[t])+\sum_{k\neq i^*}W_{jk}\vf_{j,t}(x_k[t])\\
&\geq W_{ji^*}\vf_{j,t}(x_{\max})+\sum_{k\neq i^*}W_{jk}\vf_{j,t}(x_{\min})\tag{*}\\
&\geq\vth\vf_{j,t}(x_{\max})+(1-\vth)\vf_{j,t}(x_{\min}).
\end{align*}
Hence, after one step, $x_j$ can be lower bounded by
\begin{align*}
x_j[t+1]&\geq x_{\min}+\frac{1}{L_+}\left[\vf_{j,t}(x_j[t+1])-\vf_{j,t}(x_{\min})\right]\\
&\geq x_{\min}+\frac{1}{L_+}\cdot\vth\left[\vf_{j,t}(x_{\max})-\vf_{j,t}(x_{\min})\right]\\
&\geq x_{\min}+\frac{\vth}{L_+L_-}(x_{\max}-x_{\min}),
\end{align*}
where the first inequality follows from the Lipschitz continuity of $\vf_{j,t}$, the second inequality follows from $(*)$, and the third inequality follows from the Lipschitz continuity of $\vf_{j,t}^{-1}$. Similarly, after two steps, for any node $j'$ whose shortest path to $i^*$ is within two hops, we have
\[
x_{j'}[t+2]-x_{\min}\geq\frac{\vth}{L_+L_-}\left(x_j[t+1]-x_{\min}\right)\geq\left(\frac{\vth}{L_+L_-}\right)^2(x_{\max}-x_{\min})
\]
as it will be averaged with one of $j$ which is $i^*$'s direct neighbor. We can conclude that after $D$ steps, for any node $j$ in the network, we will have
\[
x_j[t+D]\geq x_{\min}+\left(\frac{\vth}{L_+L_-}\right)^D(x_{\max}-x_{\min}),
\]
and by a symmetric argument
\[
x_j[t+D]\leq x_{\max}-\left(\frac{\vth}{L_+L_-}\right)^D(x_{\max}-x_{\min});
\]
thus, the range of $\XB[t+D]$ also shrinks with a factor smaller than $1$ in comparison to the range of $\XB[t]$:
\[
sp(\XB[t+D])\leq\left[1-2\left(\frac{\vth}{L_+L_-}\right)^D\right](x_{\max}-x_{\min})=\left[1-2\left(\frac{\vth}{L_+L_-}\right)^D\right]sp(\XB[t]).
\]
Note that the factor $\rho\triangleq\left[1-2\left(\frac{\vth}{L_+L_-}\right)^D\right]\in[0,1)$ because of the following two reasons. First, $\vth$ is defined as the smallest non-zero entry in $\WB$; we cannot have $1$ in $\WB$ since that will imply the graph is not connected, so $\vth$ will be the largest when all rows in $\WB$ have support equal to $2$ and all entries are either $0$ or $\frac{1}{2}$; hence, $0<\vth\leq\frac{1}{2}$. On the other hand, we have $L_+L_-\geq1$, where equality is attained only when all the transformations are linear. The factor $\rho$ only equals to $0$ in the trivial case of $\vth=\frac{1}{2}$, $D=1$, and $L_+L_-=1$; otherwise, it is positive but strictly less than $1$. The consensus scheme is thus a subprojection as the range of $\XB$ converges to $0$ as its components asymptotically agree.

Now we take the descent process into account. Note that the proof of Proposition 9 Part (a) in \cite{Scutari_NEXT_2016} still follows as it only involves the consensus scheme of the $\yB$ variable, which remains unchanged in our algorithm; the descent vector $\|\Dl\xB_i^{inx}[t]\|=\|\xB_i^{inx}[t]-\xB_i[t]\|\leq c_1$ is thus bounded by some constant. To prove \pref{thm:3.7}, we only have to show Proposition 9 Part (b) in \cite{Scutari_NEXT_2016}, i.e. the asymptotic agreement of $\xB_i$'s, with the nonlinear consensus scheme of $\xB$. Recall that in \pref{alg:3.1}, we have $\zB_i[t]=\xB_i[t]+\al[t]\Dl\xB_i^{inx}[t]$ and $\xB_i[t+1]=\vf_{i,t}^{-1}\left(\sum_{j=1}^nW_{ij}\vf_{i,t}(\zB_j[t])\right)$; expanding the result gives
\begin{align*}
&\eqsp\left|\vf_{i,t,l}^{-1}\left(\sum_{j=1}^nW_{ij}\vf_{i,t,l}(x_{j,l}[t])\right)-x_{i,l}[t+1]\right|\\
&\leq L_-\left|\sum_{j=1}^nW_{ij}\vf_{i,t,l}(x_{j,l}[t])-\vf_{i,t,l}(x_{i,l}[t+1])\right|\tag{Lipschitz continuity of $\vf_{i,t,l}^{-1}$}\\
&=L_-\left|\sum_{j=1}^nW_{ij}\vf_{i,t,l}(z_{j,l}[t]-\al[t]\Dl x_{j,l}^{inx}[t])-\vf_{i,t,l}(x_{i,l}[t+1])\right|\\
&\leq L_-\left|\sum_{j=1}^nW_{ij}\vf_{i,t,l}(z_{j,l}[t])-\vf_{i,t,l}(x_{i,l}[t+1])\right|\\
&\qquad+L_+L_-\sum_{j=1}^nW_{ij}\al[t]\left|\Dl x_{j,l}^{inx}[t]\right|\tag{Lipschitz continuity of $\vf_{i,t,l}$}\\
&=0+L_+L_-\sum_{j=1}^nW_{ij}\al[t]\left|\Dl x_{j,l}^{inx}[t]\right|
\leq c_2\al[t]
\end{align*}
for some fixed constant $c_2$ for dimension $l$. In the worst scenario, the added $\al$'s always go against the consensus direction, so that in the previous projection argument we have
\[
x_j[t+1]\geq x_{\min}+\frac{\vth}{L_+L_-}(x_{\max}-x_{\min})-c_2\al[t]
\]
for $i^*$'s direct neighbors and
\[
x_{j'}[t+2]\geq x_{\min}+\left(\frac{\vth}{L_+L_-}\right)^2(x_{\max}-x_{\min})-\frac{c_2\vth}{L_+L_-}\al[t]-c_2\al[t+1]
\]
for $i^*$'s $2$-hop neighbors, etc. After $D$ steps, we have the range of $\XB$ shrunk by
\[
sp(\XB[t+D])\leq\rho\cdot sp(\XB[t])+2c_2\sum_{s=0}^{D-1}\hat{\rho}^{D-1-s}\al[t+s]
\]
where $\hat{\rho}=\frac{\vth}{L_+L_-}\in(0,\frac{1}{2})$. Further iterating $t$ backwards to $t=0$ gives
\beq\label{eq:B.1}
sp(\XB[t])\leq\rho^{\lfloor t/D\rfloor}\cdot sp(\XB[0])+c_3\sum_{s=1}^t\hat{\rho}^{t-s}\al[s-1]\overset{t\ra\infty}{\longrightarrow}0.
\eeq
With the range converging to $0$, the variables from all the nodes asymptotically agree. Note that \pref{eq:B.1} holds for any dimension $l\in[d]$ and $\XB_l[t]=[x_{1,l}[t]\es\cdots\es x_{n,l}[t]]^\top$; therefore, it holds for the entire ensemble $\XB[t]\triangleq[\xB_1^\top[t]\es\cdots\es \xB_n^\top[t]]^\top$, with $sp(\XB[t])=d\cdot\max_{l\in[d]}sp(\XB_l[t])$ defined in \pref{dff:B.1}. Moreover, since the deviation from the mean $\|\xB_\perp[t]\|=\|\XB[t]-\1_n\otimes\xBb[t]\|\leq \|sp(\XB[t])\|$ will be bounded by the range, we established Eq. (62) in \cite{Scutari_NEXT_2016}. Also, since the form of \pref{eq:B.1} nearly coincides with that of Eq. (77) from \cite{Scutari_NEXT_2016}, Eq. (63) and Eq. (64) can be obtained in the same way as in \cite{Scutari_NEXT_2016}, and we have reestablished Part (b) of Proposition 9 from \cite{Scutari_NEXT_2016}.

\section{Proof of \pref{thm:3.8}}\label{app:4}\hypertarget{pf:thm:3.6}{}
We follow a framework similar to that in \cite{Yuan_DGDConv_2016} and \cite{Wei_MultCons_2018}, and the classical analysis of gradient descent \cite{Boyd_CvxOpt_2004}. We begin with one lemma and its corollary to show that the gradients are bounded from the gradient of the mean.

\begin{lem}[Bounded deviation from mean]\label{lem:B.3}
For any $i\in[n]$ and $t$, with a constant learning rate $\al[t]=\al$, we have
\beq\label{eq:B.2}
\|\xB_i[t]-\xBb[t]\|\leq sp(\XB[t])\leq \rho^{\lfloor t/D\rfloor}\cdot sp(\XB[0])+\frac{2\al BL_+L_-}{1-\hat{\rho}},
\eeq
where $\hat{\rho}=\frac{\vth}{L_+L_-}\in(0,\frac{1}{2})$, $\rho=1-2\hat{\rho}^D\in[0,1)$, and $D=diam(\Gc)$.
\end{lem}
\begin{proof}
To show this, notice that the process of DGD with nonlinear transformation (\pref{alg:3.2}) is actually very similar to the process of NEXT with nonlinear transformation (\pref{alg:3.1}); the derivation of \pref{thm:3.7}, particularly \pref{lem:B.2}, can therefore be reused here. Recall that in \pref{alg:3.1} at time $t$ for node $i$, the variable update is
\[
\xB_i[t+1]=\vf_{i,t}^{-1}\left(\sum_{j=1}^nW_{ij}\vf_{i,t}\big(\xB_j[t]+\al[t]\Dl\xB_j^{inx}[t]\big)\right),
\]
while in \pref{alg:3.2} the variable update is
\[
\xB_i[t+1]=\vf_{i,t}^{-1}\left(\sum_{j=1}^nW_{ij}\vf_{i,t}\big(\xB_j[t]-\al[t]\gd f_j(\xB_j[t])\big)\right).
\]
In \pref{alg:3.1}, the boundedness of $\Dl\xB_i^{inx}[t]$ is already established in Proposition 9 Part (a) in \cite{Scutari_NEXT_2016} with any divergent but square summable schedule of $\al[t]$. Similarly, for DGD, for a fixed learning constant $\al[t]=\al$, it is shown that with strongly convex smooth objective functions of the nodes, the gradients $\gd f_i(\xB_i[t])$'s will remain bounded. This result has to be reestablished now after adding transformations; however, for simplicity, we will just apply the assumption that the gradient $\gd f_i$ is uniformly bounded for all $i\in[n]$, i.e. $\|\gd f_i\|\leq B$.

In the proof of \pref{thm:3.7}, we show that every local averaging will move variables to each other at least by $\hat{\rho}=\frac{\vth}{L_+L_-}\in(0,\frac{1}{2})$ while being dragged apart by $c_2\al[t]$ in the worst case; the net effect is in every $D=diam(\Gc)$ steps, $sp(\XB[t])$ will be shrunk by $\rho=1-2\hat{\rho}^D\in[0,1)$ and will have an additional term $2c_2\sum_{s=0}^{D-1}\hat{\rho}^{D-1-s}\al[t+s]$. Here, the boundedness of $\Dl\xB_j^{inx}[t]$ is replaced by the boundedness of $\gd f_i$, so that similar to \pref{eq:B.1} we have
\beq\label{eq:B.3}
sp(\XB[t])\leq\rho^{\lfloor t/D\rfloor}\cdot sp(\XB[0])+2BL_+L_-\sum_{s=1}^t\hat{\rho}^{t-s}\al[s-1].
\eeq
With $\al[t]=\al$ being constant and individual node's deviation from the mean bounded by the range
\[
\|\xB_i[t]-\xBb[t]\|\leq\|\xB_\perp[t]\|=\|\XB[t]-\1_n\otimes\xBb[t]\|\leq sp(\XB[t]),
\]
we have \pref{eq:B.2} for all $t$ and $i$.
\end{proof}

\begin{cor}[Bounded gradient deviation from gradient of mean]\label{cor:B.4}
For any $i\in[n]$ and $t$, with a constant learning rate $\al[t]=\al$ and all objective functions $f_i$'s being smooth, i.e. $\|\gd f_i(\xB_1)-\gd f_i(\xB_2)\|\leq L_f\|\xB_1-\xB_2\|$ for any $i\in[n]$ and $\xB_1,\xB_2\in\R^d$ and some constant $L_f>0$, we have
\beq\label{eq:B.4}
\|\gB[t]-\gBb[t]\|\leq L_f\left[\rho^{\lfloor t/D\rfloor}\cdot sp(\XB[0])+\frac{2\al BL_+L_-}{1-\hat{\rho}}\right],
\eeq
where $\hat{\rho}=\frac{\vth}{L_+L_-}\in(0,\frac{1}{2})$, $\rho=1-2\hat{\rho}^D\in[0,1)$, $D=diam(\Gc)$, $\gB[t]\triangleq\frac{1}{n}\sum_{i=1}^n\gd f_i(\xB_i[t])$ and $\gBb[t]\triangleq\frac{1}{n}\sum_{i=1}^n\gd f_i(\xBb[t])$.
\end{cor}
\begin{proof}
The result simply follows from the Lipschitz continuity of $\gd f_i$'s and \pref{lem:B.3}, as
\[
\|\gB[t]-\gBb[t]\|\leq\frac{1}{n}\sum_{i=1}^n\left\|\gd f_i(\xB_i[t])-\gd f_i(\xBb[t])\right\|\leq\frac{L_f}{n}\sum_{i=1}^n\left\|\xB_i[t]-\xBb[t]\right\|.
\]
\end{proof}

Since all objective functions $f_i$'s have Lipschitz continuous gradient with constant $L_f$, their sum $F=\sum_if_i$ has Lipschitz continuous gradient with constant $nL_f$. Denote $\Dl\xBb[t]\triangleq\frac{1}{n}\sum_{i=1}^n(\xB_i[t+1]-\xB_i[t])$. Applying the descent lemma to get
\bal\label{eq:B.5}
&\eqsp F(\xBb[t+1])-F(\xBb[t])\\
&\leq\gd F(\xBb[t])^\top(\xBb[t+1]-\xBb[t])+\frac{nL_f}{2}\|\xBb[t+1]-\xBb[t]\|^2\\
&=n\gBb[t]^\top\Dl\xBb[t]+\frac{nL_f}{2}\|\Dl\xBb[t]\|^2\\
&=-n\al\|\gBb[t]\|^2+n\gBb[t]^\top(\Dl\xBb[t]+\al\gBb[t])+\frac{nL_f}{2}\|\Dl\xBb[t]+\al\gBb[t]-\al\gBb[t]\|^2\\
&=-n\al\left(1-\frac{\al L_f}{2}\right)\|\gBb[t]\|^2+n(1-\al L_f)\gBb[t]^\top(\Dl\xBb[t]+\al\gBb[t])+\frac{nL_f}{2}\|\Dl\xBb[t]+\al\gBb[t]\|^2\\
&\leq-n\left(\al-\frac{\al^2L_f}{2}-\frac{c_4(1-\al L_f)}{2}\right)\|\gBb[t]\|^2+\frac{n}{2}\left(L_f+\frac{(1-\al L_f)}{c_4}\right)\|\Dl\xBb[t]+\al\gBb[t]\|^2
\eal
for any constant $c_4>0$, where the last inequality follows from the perfect square formula
\begin{align*}
&\eqsp n(1-\al L_f)\gBb[t]^\top(\Dl\xBb[t]+\al\gBb[t])\\
&\leq\frac{c_4n(1-\al L_f)}{2}\|\gBb[t]\|^2+\frac{n(1-\al L_f)}{2c_4}\|\Dl\xBb[t]+\al\gBb[t]\|^2.
\end{align*}
Taking $c_4=\al$, \pref{eq:B.5} reduces to
\beq\label{eq:B.6}
F(\xBb[t+1])\leq F(\xBb[t])-\frac{n\al}{2}\|\gBb[t]\|^2+\frac{n}{2\al}\|\Dl\xBb[t]+\al\gBb[t]\|^2.
\eeq

In DGD (with nonlinear transformation averaging) we let $\zB_i[t]=\xB_i[t]+\al\gd f_i(\xB_i[t])$ be the result of local optimization at node $i$, so that the new iterate is $\xB_i[t+1]=\vf_{i,t}^{-1}\left(\sum_{j=1}^nW_{ij}\vf_{i,t}(\zB_j[t])\right)$. The term $\|\Dl\xBb[t]+\al\gBb[t]\|$ in \pref{eq:B.6} can then be bounded by 
\allowdisplaybreaks
\begin{align*}
&\eqsp\|\Dl\xBb[t]+\al\gBb[t]\|\\
&\leq\|\al\gBb[t]-\al\gB[t]\|+\|\Dl\xBb[t]+\al\gB[t]\|\\
&=\al\|\gB[t]-\gBb[t]\|+\left\|\frac{1}{n}\sum_{i=1}^n\xB_i[t+1]-\frac{1}{n}\sum_{i=1}^n\xB_i[t]+\frac{\al}{n}\sum_{i=1}^n\gd f_i(\xB_i[t])\right\|\\
&=\al\|\gB[t]-\gBb[t]\|+\frac{1}{n}\left\|\sum_{i=1}^n\left[\vf_{i,t}^{-1}\left(\sum_{j=1}^nW_{ij}\vf_{i,t}(\zB_j[t])\right)-\zB_i[t]\right]\right\|\\
&\leq\al\|\gB[t]-\gBb[t]\|+\frac{1}{n}\sum_{l=1}^d\left|\sum_{i=1}^n\left[\vf_{i,t,l}^{-1}\left(\sum_{j=1}^nW_{ij}\vf_{i,t,l}(z_{j,l}[t])\right)-z_{i,l}[t]\right]\right|.\tag{since $\|\cdot\|_2\leq\|\cdot\|_1$}\\\numberthis\label{eq:B.7}
\end{align*}

We now use Taylor's theorem to expand $\vf$ and $\vf^{-1}$ with linear approximation (first order Taylor polynomial) and a quadratic remainder. For shorthand of notation, we will omit the dimension index $l$ and the time index $t$; it is understood that the derivation works for all $l$ and $t$. We first expand $\vf_i(z_j)$ centered at $\zb_i^W\triangleq\sum_{k=1}^nW_{ik}z_k$:
\[
\vf_i(z_j)=\vf_i(\zb_i^W)+\vf_i'(\zb_i^W)(z_j-\zb_i^W)+\frac{\vf_i''(\xi_{ij}^W)}{2}(z_j-\zb_i^W)^2
\]
for some number $\xi_{ij}^W$ between $\zb_i^W$ and $z_j$. After being weighted-summed over $W_{ij}$, the linear term cancels out
\begin{align*}
&\eqsp\sum_{j=1}^nW_{ij}\vf_i(z_j)\\
&=\vf_i(\zb_i^W)\cdot\sum_{j=1}^nW_{ij}+\vf_i'(\zb_i^W)\left(\sum_{j=1}^nW_{ij}z_j-\zb_i^W\cdot\sum_{j=1}^nW_{ij}\right)+\sum_{j=1}^nW_{ij}\frac{\vf_i''(\xi_{ij}^W)}{2}(z_j-\zb_i^W)^2\\
&=\vf_i(\zb_i^W)+\sum_{j=1}^nW_{ij}\frac{\vf_i''(\xi_{ij}^W)}{2}(z_j-\zb_i^W)^2,
\end{align*}
where $\sum_{j=1}^nW_{ij}=1$ as $\WB$ is row stochastic. Let $Q_{ij}^W\triangleq\sum_{j=1}^nW_{ij}\frac{\vf_i''(\xi_{ij}^W)}{2}(z_j-\zb_i^W)^2$ denote the quadratic remainder term. Next, we expand $\vf_i^{-1}\left(\sum_{j=1}^nW_{ij}\vf_i(z_j)\right)$ centered at $\vf_i(\zb_i^W)$:
\begin{align*}
&\eqsp\vf_i^{-1}\left(\sum_{j=1}^nW_{ij}\vf_i(z_j)\right)\\
&=\vf_i^{-1}\big(\vf_i(\zb_i^W)\big)+{\vf_i^{-1}}'\big(\vf_i(\zb_i^W)\big)Q_{ij}^W+\frac{{\vf_i^{-1}}''(\vsg_{ij}^W)}{2}{Q_{ij}^W}^2\\
&=\zb_i^W+{\vf_i^{-1}}'\big(\vf_i(\zb_i^W)\big)Q_{ij}^W+\frac{{\vf_i^{-1}}''(\vsg_{ij}^W)}{2}{Q_{ij}^W}^2
\end{align*}
for some number $\vsg_{ij}^W$ between $\vf_i(\zb_i^W)$ and $\sum_{j=1}^nW_{ij}\vf_i(z_j)$.
With the expansions, the term inside the absolute value sign in \pref{eq:B.7} then becomes
\begin{align*}
&\eqsp\sum_{i=1}^n\left[\vf_{i,t,l}^{-1}\left(\sum_{j=1}^nW_{ij}\vf_{i,t,l}(z_{j,l}[t])\right)-z_{i,l}[t]\right]\\
&=\sum_{i=1}^n\left[\sum_{k=1}^nW_{ik}z_{k,l}[t]+{\vf_{i,t,l}^{-1}}'\big(\vf_{i,t,l}(\zb_{i,l}^W[t])\big)Q_{ij,l}^W[t]+\frac{{\vf_{i,t,l}^{-1}}''(\vsg_{ij,l}^W[t])}{2}{Q_{ij,l}^W[t]}^2-z_{i,l}[t]\right]\\
&=\sum_{i=1}^n\left[{\vf_{i,t,l}^{-1}}'\big(\vf_{i,t,l}(\zb_{i,l}^W[t])\big)Q_{ij,l}^W[t]+\frac{{\vf_{i,t,l}^{-1}}''(\vsg_{ij,l}^W[t])}{2}{Q_{ij,l}^W[t]}^2\right]+\sum_{k=1}^nz_{k,l}[t]\cdot\sum_{i=1}^nW_{ik}-\sum_{i=1}^nz_{i,l}[t]\\
&=\sum_{i=1}^n\left[{\vf_{i,t,l}^{-1}}'\big(\vf_{i,t,l}(\zb_{i,l}^W[t])\big)Q_{ij,l}^W[t]+\frac{{\vf_{i,t,l}^{-1}}''(\vsg_{ij,l}^W[t])}{2}{Q_{ij,l}^W[t]}^2\right].\numberthis\label{eq:B.8}
\end{align*}
Recall that the magnitudes of the first and second derivatives of $\vf^{-1}$ (resp. $\vf$) are assumed to be uniformly bounded by $L_-$ (resp. $L_+$). In addition, the quadratic remainder term $Q_{ij,l}^W[t]$ can be bounded by
\beq\label{eq:B.9}
\left|Q_{ij,l}^W[t]\right|\triangleq\left|\sum_{j=1}^nW_{ij}\frac{\vf_{i,l}''(\xi_{ij,l}^W[t])}{2}(z_{j,l}[t]-\zb_{i,l}^W[t])^2\right|\leq\frac{L_+}{2}sp(\ZB_l[t])^2.
\eeq
Substituting \pref{eq:B.8} and \pref{eq:B.9} back into \pref{eq:B.7}, we get
\allowdisplaybreaks[1]
\begin{align*}
&\eqsp\|\Dl\xBb[t]+\al\gBb[t]\|\\
&\leq\al\|\gB[t]-\gBb[t]\|+\frac{1}{n}\sum_{l=1}^d\left[\frac{nL_+L_-}{2}sp(\ZB_l[t])^2+\frac{nL_+^2L_-}{8}sp(\ZB_l[t])^4\right]\\
&\leq\al\|\gB[t]-\gBb[t]\|+\frac{L_+L_-}{2}sp(\ZB[t])^2+\frac{L_+^2L_-}{8}sp(\ZB[t])^4\\
&\leq\al L_f\left[\rho^{\lfloor t/D\rfloor}\cdot sp(\XB[0])+\frac{2\al BL_+L_-}{1-\hat{\rho}}\right]+\frac{L_+L_-}{2}\big[sp(\XB[t])+\al B\big]^2+\frac{L_+^2L_-}{8}\big[sp(\XB[t])+\al B\big]^4\tag{by \pref{eq:B.4}}
\\
&\leq\al L_f\left[\rho^{\lfloor t/D\rfloor}\cdot sp(\XB[0])+\frac{2\al BL_+L_-}{1-\hat{\rho}}\right]+\frac{L_+L_-}{2}\left[\rho^{\lfloor t/D\rfloor}\cdot sp(\XB[0])+\frac{2\al BL_+L_-}{1-\hat{\rho}}+\al B\right]^2\\
&\qquad+\frac{L_+^2L_-}{8}\left[\rho^{\lfloor t/D\rfloor}\cdot sp(\XB[0])+\frac{2\al BL_+L_-}{1-\hat{\rho}}+\al B\right]^4
\tag{by \pref{eq:B.2}}\\
&=O(\al^2)
\\\numberthis\label{eq:B.10}
\end{align*}
when $t$ is large and $\al$ is small. This is because $\rho^{\lfloor t/D\rfloor}\cdot sp(\XB[0])\longrightarrow0$ when $t\ra\infty$, and all the variables other than $\al$ are just fixed constants; after removing  the terms containing $\rho^{\lfloor t/D\rfloor}\cdot sp(\XB[0])$, the remaining terms are either $O(\al^2)$ or $O(\al^4)$, and the former dominates.

Meanwhile, as $f_i$'s are strongly convex with constant $\nu$, it follows that $F$ is strongly convex with constant $n\nu$, so that
\[
F(\xB^*)\geq F(\xBb[t])-\frac{1}{2n\nu}\|\gd F(\xBb[t])\|^2=F(\xBb[t])-\frac{n}{2\nu}\|\gBb[t]\|^2,
\]
which gives
\beq\label{eq:B.11}
-\frac{n\al}{2}\|\gBb[t]\|^2\leq-\al\nu\big[F(\xBb[t])-F(\xB^*)\big].
\eeq
Substituting \pref{eq:B.10} and \pref{eq:B.11} into \pref{eq:B.6} gives
\bal\label{eq:B.12}
&\eqsp F(\xBb[t+1])-F(\xB^*)\\
&\leq(1-\al\nu)\big[F(\xBb[t])-F(\xB^*)\big]+\frac{n}{2\al}\big[O(\al^2)\big]^2\\
&=\bar{\rho}\big[F(\xBb[t])-F(\xB^*)\big]+O(\al^3)\\
&\leq\bar{\rho}^{t+1}\big[F(\xBb[0])-F(\xB^*)\big]+\sum_{s=0}^t\bar{\rho}^{t-s}O(\al^3)\\
&\leq\bar{\rho}^{t+1}\big[F(\xBb[0])-F(\xB^*)\big]+\frac{1}{1-\bar{\rho}}O(\al^3)\\
&=\bar{\rho}^{t+1}\big[F(\xBb[0])-F(\xB^*)\big]+\frac{1}{\al\nu}O(\al^3)\\
&=\bar{\rho}^{t+1}\big[F(\xBb[0])-F(\xB^*)\big]+O(\al^2),
\eal
where the factor $\bar{\rho}\triangleq1-\al\nu$ is in $[0,1)$ by taking any $\al<\frac{1}{L_f}$ so that $\al\nu<\frac{\nu}{L_f}\in(0,1]$. The second inequality follows from repeatedly applying the inequality backwards from $t$ to $0$. In addition, the strong convexity of $F$ implies
\beq\label{eq:B.13}
\frac{n\nu}{2}\|\xBb[t+1]-\xB^*\|^2\leq F(\xBb[t+1])-F(\xB^*).
\eeq

\section{Proof of \pref{thm:3.10}}\label{app:5}\hypertarget{pf:thm:3.8}{}
Notice that any point in $\dl\circ co(\{\zB_j[t]:j\in Nb(i)\})$ can be written as the convex combination of $\{\zB_j[t]:j\in Nb(i)\}$ with positive weights. Hence, given any choice of $\xB_i[n+1]$ we can find a distribution $\{W_{ij}^x[t]\}_{j=1}^n$ such that $\xB_i[t+1]=\sum_{j=1}^nW_{ij}^x[t]\zB_j[t]$, where $W_{ij}^x[t]\geq\frac{1-\dl}{|N(i)|}\geq\frac{1-\dl}{n}>0$ for $j\in Nb(i)$ and $W_{ij}^x[t]=0$ for $j\not\in Nb(i)$. Simply speaking, if we group these weights to form a matrix $\WB^x[t]$, it satisfies the constraints described in \pref{sec:2.1} except it may not be column stochastic.

It is a standard result \cite{Wolfowitz_ProductMatrix_1963} that there exists a distribution $\{\zt_i\}_{i=1}^n$ such that
\[
\lim_{t\ra\infty}\WB^x[t]\WB^x[t-1]\cdots\WB^x[1]=\1_n\tilde{\zB}^\top,
\]
which can be argued as follows. From \pref{lem:B.2}, we know that applying $\WB^x[t]$ to any column vector $\vB$, the range of the vector $sp(\vB)$ will decrease at least by a factor of $1-\frac{2(1-\dl)}{n}$ if $\WB^x[t]$ is row stochastic. Since all the matrices $\WB^x[2],\WB^x[3],\dots$ are row stochastic, each column of $\WB^x[1]$ will align with $\1_n$ when multiplied by $\Pi_{s=2}^\infty\WB^x[s]$ as the range of the column converges to $0$ while the entries in the column asymptotically agree. The overall product is hence in the form of $\1_n\tilde{\zB}^\top$. Note that when all $\WB^x[t]$'s are doubly stochastic, the distribution $\tilde{\zB}$ is uniform.

The original proof presented in \cite{Scutari_NEXT_2016} can be directly used here with the following revisions: $\xBb[t]$ replaced by $\xB_c[t]=\sum_{i=1}^n\zt_i\xB_i[t]$, and all $\frac{1}{n}\1\1^\top$ related to $\xB$ including in the definition of $\JB$ should be replaced by $\1\tilde{\zB}^\top$.

\end{document}

%% file: Header.tex
\usepackage[intlimits]{amsmath} 

\usepackage{amsthm}
\usepackage{amssymb}
\usepackage{amsfonts}
\usepackage{graphicx}    
\usepackage{rotating}
\usepackage{color}
\usepackage{epsfig}
\usepackage{verbatim}
\usepackage{setspace}    

\usepackage{CJKutf8}
\usepackage{comment}
\usepackage[colorlinks=true,linkcolor=black,citecolor=black,urlcolor=blue]{hyperref}
\usepackage{mathtools}
\usepackage{subfig}
\usepackage{epstopdf}
\ifpdf
  \DeclareGraphicsExtensions{.eps,.pdf,.png,.jpg}
\else
  \DeclareGraphicsExtensions{.eps}
\fi
\usepackage{easybmat}
\usepackage{arydshln}
\usepackage{xcolor}
\usepackage{hyperref}
\usepackage{catchfilebetweentags}

\usepackage{algorithm}
\usepackage[algo2e, vlined, ruled]{algorithm2e}
\usepackage{algorithmicx}

\makeatletter
\newcommand{\eqnum}{\leavevmode\hfill\refstepcounter{equation}\textup{\tagform@{\theequation}}}
\makeatother
\newcounter{protocol}


\theoremstyle{plain}
\newtheorem{thm}{Theorem}

\newtheorem{cor}[thm]{Corollary}
\newtheorem{lem}[thm]{Lemma}
\newtheorem{fact}[thm]{Fact}

\newtheorem{dff}[thm]{Definition}

\newtheorem{ex}[thm]{Example}

\usepackage{prettyref}
\newcommand{\pref}[1]{\prettyref{#1}}
\newcommand{\savehyperref}[2]{\texorpdfstring{\hyperref[#1]{#2}}{#2}}
\newrefformat{thm}{\savehyperref{#1}{Theorem~\ref*{#1}}}
\newrefformat{prop}{\savehyperref{#1}{Proposition~\ref*{#1}}}
\newrefformat{cor}{\savehyperref{#1}{Corollary~\ref*{#1}}}
\newrefformat{lem}{\savehyperref{#1}{Lemma~\ref*{#1}}}
\newrefformat{fact}{\savehyperref{#1}{Fact~\ref*{#1}}}
\newrefformat{conj}{\savehyperref{#1}{Conjecture~\ref*{#1}}}
\newrefformat{ass}{\savehyperref{#1}{Assumption~\ref*{#1}}}
\newrefformat{ppt}{\savehyperref{#1}{Property~\ref*{#1}}}
\newrefformat{dff}{\savehyperref{#1}{Definition~\ref*{#1}}}
\newrefformat{cond}{\savehyperref{#1}{Condition~\ref*{#1}}}
\newrefformat{ex}{\savehyperref{#1}{Example~\ref*{#1}}}
\newrefformat{rmk}{\savehyperref{#1}{Remark~\ref*{#1}}}
\newrefformat{alg}{\savehyperref{#1}{Algorithm~\ref*{#1}}}
\newrefformat{ptc}{\savehyperref{#1}{Protocol~\ref*{#1}}}
\newrefformat{sec}{\savehyperref{#1}{Section~\ref*{#1}}}
\newrefformat{app}{\savehyperref{#1}{Appendix~\ref*{#1}}}
\newrefformat{fig}{\savehyperref{#1}{Figure~\ref*{#1}}}
\newrefformat{tab}{\savehyperref{#1}{Table~\ref*{#1}}}
\newrefformat{chap}{\savehyperref{#1}{Chapter~\ref*{#1}}}

\makeatletter
\DeclareTextCommand{\textprime}{\encodingdefault}{%
  \mbox{$\m@th'\kern-\scriptspace$}%
}
\makeatother

   \def\Dl{\Delta}

\def\al{\alpha}   \def\dl{\delta}
\def\ep{\epsilon}

\def\vsg{\varsigma} \def\vth{\vartheta} 
  \def\Cc{\mathcal{C}} 
\def\Ec{\mathcal{E}} \def\Fc{\mathcal{F}} \def\Gc{\mathcal{G}} 
  \def\Kc{\mathcal{K}} 
 \def\Nc{\mathcal{N}}  \def\Pc{\mathcal{P}}
   
 \def\Vc{\mathcal{V}}  
 
\def\I{\mathbb{I}}   
 \def\N{\mathbb{N}}  
 \def\R{\mathbb{R}}

 \def\JB{\mathbf{J}}  
\def\MB{\mathbf{M}}   
   
  \def\WB{\mathbf{W}} \def\XB{\mathbf{X}}
 \def\ZB{\mathbf{Z}}
 \def\bB{\mathbf{b}}  
  \def\gB{\mathbf{g}}

 \def\vB{\mathbf{v}}  \def\xB{\mathbf{x}}
\def\yB{\mathbf{y}} \def\zB{\mathbf{z}}

   \def\pb{\bar{p}}

 \def\zb{\bar{z}}
   
 \def\ft{\tilde{f}} 

 \def\zt{\tilde{z}}

\def\zs{z^*}
\def\xs{x^*}

\def\xBt{\tilde{\mathbf{x}}}

\def\xBb{\bar{\mathbf{x}}}

\def\piB{\mathbf{\pi}}
\def\piBt{\tilde{\mathbf{\pi}}}

\def\vf{\varphi}

\def\Lmax{L^{\max}}

\def\gBb{\bar{\mathbf{g}}}

\usepackage{stackengine}
\stackMath
\newcommand\tsup[2][2]{%
 \def\useanchorwidth{T}%
  \ifnum#1>1%
    \stackon[-.5pt]{\tsup[\numexpr#1-1\relax]{#2}}{\scriptscriptstyle\sim}%
  \else%
    \stackon[.5pt]{#2}{\scriptscriptstyle\sim}%
  \fi%
}

\usepackage{accents}
\newlength{\dhatheight}

\def\rU{\subset}

\def\uU{\cup}

\def\rUe{\subseteq}

\def\sm{\setminus}
\def\buU{\bigcup}
\def\bdU{\bigcap}
\def\gd{\nabla}

\def\es{\enspace}
\def\ra{\rightarrow}
\def\la{\leftarrow}

\def\1{\mathbf{1}}
\def\0{\mathbf{0}}
\def\eqsp{{\hspace{1.35em}}}
\DeclareMathOperator*{\argmin}{argmin} 
\DeclareMathOperator*{\argmax}{argmax} 

\newcommand{\rn}[1]{\textup{\lowercase\expandafter{\romannumeral#1}}}

\def\bms{\begin{bmatrix}}
\def\bme{\end{bmatrix}}
\def\beq{\begin{equation}}
\def\eeq{\end{equation}}
\def\bal{\begin{equation}\begin{aligned}}
\def\eal{\end{aligned}\end{equation}}
\def\bals{\begin{equation*}\begin{aligned}}
\def\eals{\end{aligned}\end{equation*}}

\newcommand\numberthis{\addtocounter{equation}{1}\tag{\theequation}}